\pdfoutput=1
 \documentclass[11pt]{amsart}
\usepackage{amsmath,amssymb}
\usepackage{graphicx}
\usepackage{amsfonts,amscd,latexsym,epsfig, psfrag,nicefrac,color}

\usepackage[all]{xy}

\newcommand{\ds}{\displaystyle}
\newcommand{\nbd}{neighborhood}
\newcommand{\lra}{\longrightarrow}
\newcommand{\nf}[2]{{\nicefrac{#1}{#2}}}
\newcommand{\PD}{{\rm PD}}

\newcommand{\Cone}{{\rm Cone}}

\newcommand{\Ppp}{{\mathcal P}}

\newcommand{\Gr}{{\rm Gr}}

\newcommand{\TT}{{\mathbb T}}

\newcommand{\Aa}{{\mathcal A}}

\newcommand{\Ll}{{\mathcal L}}

\newcommand{\ind}{{\rm ind\,}}

\newcommand{\Coker}{{\rm Coker\,}}

\renewcommand{\Hat}{\widehat}

\newcommand{\Jac}{{\rm Jac}\,}

\newcommand{\Cc}{{\mathcal C}}

\newcommand{\Ff}{{\mathcal F}}
\newcommand{\Jj}{{\mathcal J}}

\newcommand{\im}{{\rm im\,}}
\newcommand{\less}{{\smallsetminus}}

\newcommand{\p}{{\partial}}
\newcommand{\al}{{\alpha}}

\newcommand{\Om}{{\Omega}}
\newcommand{\om}{{\omega}}
\newcommand{\eps}{{\varepsilon}}
\newcommand{\vareps}{{\epsilon}}
\newcommand{\de}{{\delta}}

\newcommand{\ga}{{\gamma}}

\newcommand{\ka}{{\kappa}}
\newcommand{\la}{{\lambda}}
\newcommand{\La}{{\Lambda}}
\newcommand{\si}{{\sigma}}
\newcommand{\Si}{{\Sigma}}

\newcommand{\Mm}{{\mathcal M}}
\newcommand{\Ss}{{\mathcal S}}
\newcommand{\Tt}{{\mathcal T}}
\newcommand{\oMm}{{\overline {\Mm}}}
\newcommand{\ov}{\overline}

\newcommand{\id}{{\rm id}}

\newcommand{\coker}{{\rm Coker\,}}

\newcommand{\Ee}{{\mathcal E}}
\newcommand{\Ii}{{\mathcal I}}

\newcommand{\Q}{{\mathbb Q}}
\newcommand{\R}{{\mathbb R}}
\newcommand{\C}{{\mathbb C}}

\newcommand{\Z}{{\mathbb Z}}
\newcommand{\CP}{{\mathbb{CP}}}

\newcommand{\Nn}{{\mathcal N}}
\newcommand{\Pp}{{\mathcal P}}

\newcommand{\bx}{{\bf x}}

\newtheorem{theorem}{Theorem}[subsection]
\newtheorem{thm}[theorem]{Theorem}

\newtheorem{cor}[theorem]{Corollary}
\newtheorem{lemma}[theorem]{Lemma}

\newtheorem{prop}[theorem]{Proposition}

\newtheorem{defn}[theorem]{Definition}
\newtheorem{example}[theorem]{Example}
\newtheorem{cond}[theorem]{Condition}

\newtheorem{fact}[theorem]{Fact}

\newtheorem{remark}[theorem]{Remark}
\newtheorem{rmk}[theorem]{Remark}
\newtheorem{question}[theorem]{Question}
\newtheorem{quest}[theorem]{Question}
\numberwithin{figure}{subsection}
\numberwithin{equation}{subsection}

\newcommand{\MS}{{\medskip}}

\newcommand{\NI}{{\noindent}}

\newcommand{\gr}{\operatorname{graph}}

\newenvironment{itemlist}
   { \begin{list} {$\bullet$}
         {  \setlength{\itemsep}{.5ex} \setlength{\leftmargin}{2.5ex} } }
   { \end{list} }
\newenvironment{itemlistt}
   { \begin{list} {$\bullet$}
         {  \setlength{\itemsep}{.5ex} \setlength{\leftmargin}{3.5ex} } }
   { \end{list} }

 \title{\vspace*{-1cm}Nongeneric $J$-holomorphic curves and singular inflation}
 \author{Dusa McDuff}\thanks{First author partially supported by NSF grants DMS 0905191 and 1308669}
\address{Department of Mathematics,
 Barnard College, Columbia University
New York, USA}
\email{dusa@math.columbia.edu}
 \author{Emmanuel Opshtein}\thanks{Second author partially supported by the grant ANR-116JS01-010-01}\address{IRMA, Universit\'e de Strasbourg, FRANCE}
\email{opshtein@math.unistr.fr}

\date{November 27, 2013}

\begin{document}

\begin{abstract} This paper investigates the geometry of a
symplectic $4$-manifold
 $(M,\om)$ relative to a  $J$-holomorphic normal crossing divisor $\Ss$.
  Extending work by Biran (in {\it Invent. Math.} 1999), we give conditions
  under which
 a homology class $A\in H_2(M;\Z)$ with nontrivial Gromov invariant
 has an embedded $J$-holomorphic representative for some $\Ss$-compatible $J$.   
 This holds for example if the class $A$ can be represented by an embedded sphere, or 
 if the components of $\Ss$ are spheres with self-intersection $-2$.
We also show that inflation relative to $\Ss$ is  always possible, a result that 
allows one  
to calculate the relative symplectic cone. 
It also has important 
applications to various embedding problems, for example of ellipsoids or Lagrangian submanifolds.
\end{abstract}

\maketitle

\tableofcontents
\section{Introduction}

\subsection{Overview}

Inflation is  
an important tool for understanding
 symplectic embeddings in dimension $4$. Combined with Taubes--Seiberg--Witten theory, it provides a powerful method to study these embedding problems, especially in so-called rational or ruled symplectic manifolds. Non exhaustive references for ball packings are \cite{mcpo,Bir,Mcd}. In recent years, these results have been extended in several directions \cite{mcsc,buhi,mc-hof}, building on a work of the first author on ellipsoid embeddings \cite{Mce}. 
Unfortunately, this paper contains a gap, which we describe briefly now. The classical inflation method requires
that one finds  an embedded symplectic curve in a given homology class $A$, that intersects 
some fixed divisor transversally and positively. 
When this divisor is regular in the sense of $J$-holomorphic curve theory ---
 as in the case of ball packings, where it is an exceptional divisor ---, this embedded representative of $A$ is found via Taubes' work on pseudo-holomorphic curves in dimension $4$. For ellipsoid embeddings however, these divisors are not regular, so the relevant almost complex structures are not generic, and the theory must be adapted, which was not done in \cite{Mce}. This discussion raises the following general question:

\begin{question}\label{Q:emb}
Given 
a homology class $A\in H_2(M)$ in a symplectic $4$-manifold, with embedded $J$-representatives for a generic set of $J$, are there natural conditions that ensures that $A$ also has an embedded $J$-representative, where $J$ 
is now prescribed on some fixed divisor $\Ss$ ?
\end{question}

In fact, as realized by Li--Usher~\cite{LU}, a complete answer to this question is not needed for inflation: non-embedded representatives can also be used to inflate, and,
as was shown in \cite{Merr}, this suffices to deal with the main gap in \cite{Mce}.
 
 \begin{question}\label{Q:infnod} To which extent can nodal curves replace embedded ones as far as inflation is concerned ?
\end{question}
The present paper is concerned with these two questions. The main results  are Theorem~\ref{thm:1} that gives conditions under which
a class $A$ has an embedded  $J$-holomorphic representative for $\Ss$-adapted $J$ and
Theorem~\ref{thm:inf} which explains that nodal curves can be used for inflation in $1$-parameter families  relative to $\Ss$ (leading to a relative version of  \lq\lq deformation implies isotopy" \cite{Mcd}).

  \subsection{Main results}\label{ss:main}

We assume throughout that $(M,\om)$ is a closed symplectic $4$-manifold.
We first discuss the kind of singular sets $\Ss$ we consider, and give a local model for their \nbd s.
 A \nbd\ $\Nn(C)$ of a ($2$-dimensional)  symplectic submanifold $C$ can always be identified with a \nbd\ of the zero section in a holomorphic line bundle $\Ll$ over $C$ with Chern class 
$[C]\cdot [C]$. 
 For a union $\Ss=\cup\ C^{S_i}$ of submanifolds that intersect positively and $\om$-orthogonally the local model is a plumbing: we identify the standard \nbd s $\Nn(C^{S_i})$ with $\Nn(C^{S_j})$ at an intersection point 
$q\in C^{S_i}\cap C^{S_j}$
by preserving the local product structure but interchanging fiber and base. Thus each 
such $q$ has a product neighborhood $\Nn_q$, and by a local isotopy we can always arrange that 
this product structure is compatible with $\om$, i.e.  $\om\big |_{\Nn_q}$ is the sum of 
the pullbacks of its restrictions to $C^{S_i}$ and $C^{S_j}$.
We call the resulting plumbed structure on the neighborhood $\Nn(\Ss)=\cup_i\ \Nn(C^{S_i})$
the {\bf local  fibered structure}.

 \begin{defn}\label{def:sing}  A {\bf singular set} $\Ss: = C^{S_1}\cup \dots\cup C^{S_s}$ of $(M,\om)$ is a union of symplectically  embedded
curves of genus $g(S_i)$ in classes $S_1,\dots, S_s$ respectively
whose pairwise intersections are transverse and $\om$-orthogonal.
A component $C^{S_i}$ is called {\bf negative} if $(S_i)^2< 0$ and {\bf nonnegative} otherwise, and is called
{\bf regular} if $(S_i)^2 \ge g-1$.
We write $\Ss_{sing}$  (resp. $\Ss_{irreg}$) for the collection of  components that are negative 
and 
not regular (resp. not regular), and define 
$\Ii_{sing}:= \{i: C^{S_i}\in \Ss_{sing}\}$ and $\Ii_{irreg}:= \{i: C^{S_i}\in \Ss_{irreg}\}$.

We say that the {\bf  symplectic form $\om$ is  adapted to $\Ss$} if the  conditions above are satisfied and if $\om$ is compatible with the local fiber structure on some \nbd\ $\Nn(\Ss)$.

Given, a closed fibered neighborhood $\ov \Nn$ of $\Ss$ we say that 
 an $\om$-tame 
 {\bf almost complex structure $J$ is  $(\Ss,\ov\Nn)$-adapted}
 if it is integrable in $\ov\Nn$ and if each $C^{S_i}$ as well as each local projection  $\ov\Nn(C^{S_i})\to C^{S_i}$ 
is $J$-holomorphic. 
We define $\Jj(\Ss,\ov\Nn): = \Jj(\Ss,\ov \Nn,\om)$ to be the space of all  such  almost complex structures $J$.
The space of {\bf $\Ss$-adapted } almost complex structures is the union 
 $\Jj(\Ss): = \bigcup_{\ov\Nn} \Jj(\Ss,\ov\Nn)$ with
 the direct limit topology. 
\end{defn}

\NI  We suppose throughout that $\Ss$ satisfies the conditions of  Definition~\ref{def:sing}, and will call it the singular set, even though some of its components may not be in any way singular.

\begin{remark}\label{rmk:sing}\rm (i)
The regularity condition can also be written as $d(S_i)\ge 0$, where $d(S_i) : = c_1(S_i) + (S_i)^2$  is the
Seiberg--Witten degree.
 By \eqref{eq:Find},  any regular component $C^{S_i}$ can be given a $J$-holomorphic parametrization 
for some $J\in \Jj(\Ss)$ such that the linearized Cauchy--Riemann operator  is surjective. In other words, the parametrization is regular in the usual sense for $J$-holomorphic curves; cf. \cite[Chapter~3]{MS}.  On the other hand, if 
$C^{S_i}$ is not regular, this is impossible.   
Further, by 
Remark~\ref{rmk:nodaldim}~(ii),
 if $C^{S_i}$ is regular but  negative then it is an exceptional sphere.   
 Therefore $\Ss_{sing}$ consists of spheres with self-intersection $\le -2$ and 
 higher genus curves with negative self-intersection.  
\MS

 \NI (ii) The orthogonality condition~(ii) in Definition~\ref{def:sing} is purely technical.
 If all intersections are transverse and positively oriented we can always
 isotop the curves in $\Ss$ so that they intersect 
 orthogonally; cf. Proposition~\ref{prop:gp}. 
\end{remark}

 \begin{example}\label{ex:tor}  \rm  Suppose that $(M,\om,J)$ is a toric manifold whose moment polytope has  a connected chain of edges
 $\vareps_i, i=1,\dots, s,$ with   Chern numbers 
  $-k_i\le-2$. Then the inverse image $\Ss$ of this chain of edges under the moment map is a chain of spheres with respect to the natural complex structure on $M$.
Moreover the toric symplectic form is adapted to $\Ss$: in particular the 
spheres $C^{S_i}$ do intersect orthogonally.
 Another example of $\Ss$  is  a disjoint union of embedded spheres
 each with self-intersection $\le -2$.
\end{example}

Write $\Ee\subset H_2(M;\Z)$ for the 
set of classes that can be represented by exceptional spheres, 
i.e. symplectically embedded spheres with self-intersection $-1$.

\begin{defn}\label{def:*S}  
A nonzero class $A\in H_2(M;\Z)$ is said to be {\bf $\Ss$-good} if:
\begin{itemize}\item[(i)]  $\Gr(A)\ne 0$;
\item[(ii)]  if $A^2=0$ then $A$ is a primitive class;
\item[(iii)]
$A\cdot E\ge 0$ for every $E\in \Ee\less  \{A\}$; and
\item[(iv)]
$A\cdot S_i\ge 0$ for  $1\le i\le s$.
\end{itemize}
\end{defn}

\begin{example}\label{ex:*S}\rm   
As we explain in more detail in \S\ref{ss:SW},
when $M$ is rational (i.e. $S^2\times S^2$ or a blow up of $\C P^2$)  the Gromov invariant $\Gr(A)$
is nonzero
 whenever $A^2>0$, $\om(A)>0$ and the  Seiberg--Witten degree $d(A): = A^2 + c_1(A)$ is $ \ge 0$.
Thus condition (i) above is easy to satisfy. 
  Further, 
if $A\notin\Ee$ satisfies (i) and (iii) then $A^2\ge 0$.
 \end{example}

Here is a more precise version of 
Question \ref{Q:emb}. 
\begin{quest}\label{q:1} Suppose that $A$
is $\Ss$-good.  When is there 
an embedded connected curve $C^A$ in class $A$ that is $J$-holomorphic for some $J\in \Jj(\Ss)$?
\end{quest}

If $\Ss_{sing}= \emptyset$, then the answer is \lq\lq always".  Therefore the interesting case is when at
least one component of $\Ss$ is not regular.\footnote{
See Remark~\ref{rmk:nodaldim}~(i) for an explanation of the problem in analytic terms.}
So far we have not managed to answer this question  by trying to construct $C^A$ geometrically.\footnote{
In a previous version of this paper, the first author claimed to   
carry out such a construction.
However,  the second author pointed out first that 
 the complicated inductive argument had a flaw and, more seriously, 
 that some of the geometric constructions were incomplete.}
The difficulties with such a direct approach are explained  in \S\ref{ss:geom}.
Nevertheless, in various situations one can obtain a positive answer by using  numerical arguments.
In cases 
(iii) and (iv)
below the class $A$ has genus $g(A) = 0$,
where, by the adjunction formula~\eqref{adj}, $g(A): = 1 + \frac 12(A^2-c_1(A))$
 is the genus of any  embedded and connected $J$-holomorphic representative of $A$.
Our proof of 
(iv)
 adapts arguments in Li--Zhang~\cite{LZ}, while that in 
(v)
generalizes Biran~\cite[Lemma~2.2B]{Bir}.
Finally 
(ii)
follows by an easy special case  of the geometric construction that works 
because  $\Ss$ is not very singular.

\begin{thm}\label{thm:1}   Let $(M,\om)$ be a symplectic $4$-manifold with 
a singular set $\Ss$, and suppose that
$A\in H_2(M;\Z)$  is $\Ss$-good.
\begin{itemlistt}\item   In the following cases there is $J\in \Jj(\Ss)$ such that $A$ has an
embedded $J$-holomorphic representative:
\begin{itemize}\item[(i)]  $\Ss_{sing}=\emptyset$, i.e. the only components of $\Ss$ with negative square are exceptional spheres;
\item[(ii)]   $\Ss_{sing}$ consists of a single sphere with $S^2 = -k$ where $2\le k\le 4$.
\end{itemize}
\item  In the following cases there is a residual subset $\Jj_{emb}(\Ss,A)$ of $\Jj(\Ss)$ 
such that  $A$ is represented by  an embedded $J$-holomorphic curve $C^A$
for all $J\in \Jj_{emb}(\Ss,A)$:
\begin{itemize}
\item[(iii)]   $A\in \Ee$;
\item[(iv)]   $g(A): = 1 + \frac 12(A^2-c_1(A)) = 0$;
\item[(v)] 
the components of $\Ss_{irreg}$ have $c_1(S_i) = 0$ and $A$ cannot be written as 
$\underset{i\in \Ii_{irreg}}\sum \ell_i S_i$ where $\ell_i\ge 0$. 
\end{itemize} 
Moreover, any two elements $J_0, J_1\in \Jj_{emb}(\Ss, A)$ can be joined by a path 
$J_t, t\in [0,1],$  in 
$\Jj(\Ss)$  for which there is a smooth family 
 of embedded $J_t$-holomorphic $A$-curves.
\end{itemlistt}
\end{thm}

\begin{rmk}\label{rmk:1}\rm 
(i) Although we do not assume initially that $b^+_2=1$, it is well known 
that any $4$-manifold $M$ that has a class 
$A$ with $\Gr(A)\ne 0$ and $d(A)>0$ must have
$b_2^+= 1$: cf.
Fact~\ref{fact:grdb}.  
The same holds if $d(A) = g(A) = 0$ and $A\notin \Ee$: cf.  Lemma~\ref{le:recog}.
Therefore in almost all cases covered by (iv) and (v) we must have $b^+_2=1$.
\MS

\NI (ii)  If $c_1(S_i) = 0$ and $g(S_i)>0$ then $(S_i)^2\ge 0$ so that $d(S_i)\ge 0$, in other words, $C^{S_i}$ is regular.
Therefore in (iv) the components of $\Ss_{irreg}$ must be spheres.
 \MS

  \NI (iii)  
As noted in Remark~\ref{rmk:iv} below, the condition on $k$ in part (v) 
 above can 
almost surely be
 improved.  We restrict to $k\le 4$ to simplify the proof, 
 and because these are the only cases that have been applied; cf. \cite{BLW,WW}.
 \end{rmk}

In general, the issues involved in constructing a single  embedded representative 
of  a class $A$ are rather different from those involved in constructing a $1$-parameter family of embedded $J_t$-holomorphic curves for a generic path $J_t\in \Jj(\Ss)$.  In particular, as we see in Lemma~\ref{le:gp2}, 
the presence of  positive but nonregular components of $\Ss$ 
can complicate matters. 
Further, in cases (i) and (ii) of Theorem~\ref{thm:1}  we have no independent characterization 
(e.g. via Fredholm theory) of those $J\in \Jj(\Ss)$ that admit 
 embedded $A$-curves, and also cannot guarantee that there is a $1$-parameter family of embedded curves connecting
  any given pair of embedded curves. Even if we managed to include them as part of the boundary of a $1$-manifold of curves, they may well not lie in the same connected component.
 Hence,  without extra hypotheses, 
  it makes very little sense to try to construct $1$-parameter families of such curves for fixed symplectic form $\om$.   
   However,
  if we add an extra hypothesis (such as (ii) below) then we can construct such families.  
  We will prove a slightly more general result that applies when we are given a family
 $\om_t, t\in [0,1],$  of $\Ss$-adapted symplectic forms.

\begin{prop}\label{prop:1} 
 Let  $(M,\om)$ be a blowup of a rational or ruled manifold, and
let $\om_t, t\in [0,1]$, be a smooth family of $\Ss$-adapted  symplectic forms.
Suppose that
\begin{itemize}\item [(i)]  $\Ss_{sing}$  is either empty or contains one sphere of self-intersection $-k$ with $2\le k\le 4$, and 
\item[(ii)] $d(A)>0$ if $M$ is rational, and $d(A)> g + \frac n4$ if $M$ is the $n$-point blow up of a ruled surface of genus $g$. 
\end{itemize} 
Then, possibly after reparametrization with respect to $t$,  any pair $J_\al\in \Jj(\Ss, \om_\al, A), \al=0,1,$ for which $A$ has an embedded holomorphic representative  can be joined by a path 
$J_t\in \Jj(\Ss, \om_t,A)$ for which there is a smooth family 
 of embedded $J_t$-holomorphic $A$-curves.
\end{prop}

\begin{rmk}\rm  The proof of parts (iii), (iv) and (v) of Theorem~\ref{thm:1} easily extends to prove a similar statement in these cases, but without  hypothesis (ii) on $A$.
\end{rmk}

The gap in \cite{Mce} precisely consisted in the claim that every
 $\Ss$-good
class $A$ 
does have an $\Ss$-adapted embedded  representative, and as explained already,  this was used to justify 
certain inflations and hence the existence of  certain embeddings.   
Even though we still have not found an answer to Question~\ref{q:1}, as far as inflation goes one can avoid it:   as explained in \cite{Merr}, one can in fact inflate along suitable nodal curves. Thus the following holds.

\begin{lemma}\label{le:infl00}  If $A$ 
is $\Ss$-good
and $A\cdot S_i\ge 0$ for all components $S_i$ of $\Ss$, then there is a family of symplectic forms $\om_{\ka,A}$ in class $[\om]+\ka\PD(A)$, 
$\ka\ge 0$,
 that are nondegenerate on $\Ss$ and have 
$\om_{0,A} = \om$. 
\end{lemma}

This result 
(which is reproved in Lemma~\ref{le:tech1} below)
suffices to establish the existence of the desired embeddings.
  However  
 to prove their uniqueness up to isotopy one needs to inflate in $1$-parameter families, in   
other words, we need the following relative version of the \lq\lq deformation implies isotopy" result of \cite{Mcd}. 

\begin{thm} \label{thm:inf} Let $(M,\om)$ be a blow up of a rational or ruled $4$-manifold, and
let $\Ss\subset M$ satisfy the conditions of Definition~\ref{def:sing}.
 Let $\om'$ be any symplectic form  on $M$  such that the following conditions hold:
 \begin{itemize}
  \item[(a)] $[\om']=[\om]\in H^2(M)$;
 \item[(b)] there is a family of possibly noncohomologous symplectic forms $\om_t, t\in [0,1],$
 on $M$ that 
are nondegenerate on $\Ss$ and are such that $\om_0=\om$ and  $ \om_1=\om'$.
 \end{itemize}
 Then there is a family $\om_{st}, s,t\in [0,1],$ of symplectic forms such that
  \begin{itemize}
  \item  $\om_{0t}=\om_t$ for all $t$ and
 $[\om_{1t}], t\in [0,1]$ is constant;
\item $ \om_{s0}=\om$ and $\om_{s1}=\om'$ for all $s$;
\item $\om_{st}$ is nondegenerate on each component of $\Ss$.
 \end{itemize}
Moreover, if  $\om = \om'$ near $\Ss$, 
 we can arrange that all the forms 
$\om_{1t}, t\in [0,1],$ equal $\om$ near $\Ss$.
  \end{thm}
  
\begin{cor}\label{cor:inf}
Under the assumptions of Theorem \ref{thm:inf},  there is 
an isotopy $\phi_t, t\in [0,1],$ of  $M$ 
 such that $\phi_0=\id$, $\phi_1^*(\om')= \om$ and 
  $\phi_t(\Ss) = \Ss$ for all $t$.  Moreover, if $\om = \om'$ near $\Ss$ we may choose
 this isotopy to be compactly supported in $M\less \Ss$.
\end{cor}

\begin{rmk}\rm
Li--Liu show in
 \cite[Theorems~2,3]{LL}  that every manifold with $b_2^+ = 1$ has  enough nonvanishing  Seiberg--Witten invariants
 to convert any family $\om_t$ with cohomologous endpoints to an isotopy.  It is likely that 
  Proposition~\ref{prop:1} and
 Theorem~\ref{thm:inf} 
 also extends to this case since the rational/ruled hypothesis is needed only
 via Lemma~\ref{le:SW}, which guarantees conditions 
 that $\Gr(B)\ne 0$ in Propositions~\ref{prop:Z} and~\ref{prop:Zt}.
\end{rmk}

The  results on inflation can be rephrased in terms of  the {\bf relative symplectic cone}
$\Cone_\om(M,\Ss)$.  
Denote by $\Om_\om(M)$  the connected component  containing $\om$ of the
space of symplectic forms on $M$, and by   $\Om_\om(M,\Ss)$
its subset consisting of forms that are nondegenerate on $\Ss$.
Further given $a\in H^2(M)$ let  $\Om_\om(M,\Ss,a)$ be the subset of $\Om_\om(M,\Ss)$ consisting of forms in class $a$.   
Define
\begin{eqnarray}
\Cone_\om(M) : &=&  \bigl\{[\si]\ \big| \  \si\in \Om_\om(M)\bigr\}\subset  H^2(M;\R), \notag \\
 \Cone_\om(M,\Ss) : &=&  \bigl\{[\si]\ \big| \  \si\in \Om_\om(M,\Ss)\bigr\}\subset  H^2(M;\R). \notag 
\end{eqnarray}
Note that these cones are connected by definition.
If $(M,\om)$ is a blowup of a rational or ruled manifold, 
it is well known that
$$
\Cone_\om(M) = \{a\in H^2(M;\R)\ \big| \  a^2>0,\ a(E)>0 \ \forall E\in \Ee\};
$$
see  \cite{LL2} and the proof of
Proposition~\ref{prop:cone} given below.\footnote{It follows easily from Gromov--Witten theory that the set $\Ee = \Ee_{\om'}$ of all classes represented by $\om'$-symplectically embedded  $-1$ spheres is the same for all $\om'\in \Om_\om(M)$.  Therefore, this description of $\Cone_\om(M)$ makes sense.}
 In this language, Lemma~\ref{le:infl00} and Theorem~\ref{thm:inf} 
can be restated as follows.
\begin{prop}\label{prop:cone}  Let $(M,\om)$ be a blowup of a rational or ruled manifold and $\Ss$ a singular set. Then:
\begin{enumerate}
\item
$\Cone_\om(M,\Ss) = \bigl\{a\in \Cone_\om(M)\ \big| \  a(S_i)>0,\ 1\le i\le s\bigr\}$.
\item $\Om_\om(M,\Ss,a)$ is path connected.
\end{enumerate}
\end{prop}
\begin{proof}  
In (i) the left hand side is clearly contained in the right hand side.
To prove the reverse inclusion,
first notice that the set of classes represented by symplectic forms that evaluate positively on the $S_i$ is open in $H^2(M,\R)$.
Hence, if $a\in \Cone(M,\om)$, satisfies $a(S_i)>0$ $\forall i$, so does 
$a'=a-\eps[\om]$  for $\eps>0$ sufficiently small.  
Further, by perturbing $\om$, we may choose $\eps$ so that $a'\in H^2(M;\Q)$.  
Since $M$ is rational or ruled, the class $q\PD(a')$ 
is $\Ss$-good
 for $q$ sufficiently large (see Corollary \ref{cor:SW}).
Thus, by Lemma \ref{le:infl00}, the class $[\om]+\kappa qa'$ is represented by a symplectic form $\om_\kappa$ for all $\kappa>0$. 
Taking $\kappa = \frac 1{q\eps}$, we therefore obtain a symplectic form $\eps  \om_\kappa$ in class $a$.
This proves (i).
Finally, (ii)  holds  because, by definition of $\Om_\om(M,\Ss)$, any two symplectic  forms in $\Om_\om(M,\Ss,a)$ are deformation equivalent, thus isotopic by Theorem \ref{thm:inf}. 
\end{proof}

Finally, we show that
these singular inflation procedures 
combine 
with the Donaldson construction
to provide approximate asymptotic answers to 
Question \ref{q:1}.

\begin{thm}\label{thm:asympt}
Let $(M^4,\om)$ be a symplectic manifold with a singular set  $\Ss$ and 
an $\Ss$-good class $A\in H_2(M)$.
Then: 
\begin{enumerate}
\item  There is a union $\Tt$ of transversally and positively intersecting symplectic submanifolds $C^{T_1},\dots,C^{T_r}$, orthogonal to $\Ss$ and
such that $\PD(\om)= \sum_{j=1}^r \beta_jT_j$, where
$\beta_j>0$.  Further, we may take $r\le {\rm rank } H^2(M)$, and, if $[\om]$ is rational, we may take $r=1$.
\item 
For all positive $\eps_1,\dots,\eps_k \in \Q$,  
there are integers $N_0, k_0\ge 1$ such that $N_0(A+\sum \eps_i T_i)$ is integral and each class   $kN_0(A+\sum \eps_i T_i), k\ge k_0,$ is represented by an embedded $J$-curve for some $J\in \Jj(\Ss\cup\Tt)$.
\end{enumerate}
\end{thm}

\begin{cor}\label{cor:asympt} If $k$ closed balls of size $a_1,\dots,a_k\in \Q$ embed into $\CP^2$, and if $\Ss$ is any singular set  in the $k$-fold blow-up of $\CP^2$, then there is $N$ such that
the class $N(L-\sum a_iE_i)$ has an embedded $J$-representative for some $J\in \Jj(\Ss)$.
\end{cor}

\subsection{Plan of the paper.} 
Because this paper deals with nongeneric $J$, we must rework standard $J$-holomorphic curve theory, adding  quite a few rather fussy technical details.  For the convenience of the reader, we begin in
 \S 2, by surveying relevant aspects of Taubes--Seiberg--Witten theory
explaining in particular why Question \ref{q:1} has a positive answer when $\Ss=\emptyset$. 
We then describe the modifications needed when $J$ is $\Ss$-adapted.  

The next section \S\ref{s:proof} proves most cases of Theorem~\ref{thm:1}.
The basic strategy of the proof  is to represent the class $A$ by an embedded $J_\eps$-representative for sufficiently generic $J_\eps$ and let $J_\eps$ tend to some $J\in \Jj(\Ss)$. By Gromov compactness, 
we get a nodal $J$-representative for $A$, 
whose properties are investigated in  Lemmas~\ref{le:SiA1} and \ref{le:Abir}.  
To prove part~(i) of Theorem \ref{thm:1} it then suffices to amalgamate these components into a single curve, which is always possible for components with nonnegative self-intersection;
cf. Corollary~\ref{cor:gp0}.  
Since this geometric approach gets considerably more complicated when $\Ss$ has negative
 components, the proof of part (ii) of Theorem \ref{thm:1} is deferred to \S\ref{s:geom}.
  Proposition~\ref{prop:1}, which is a $1$-parameter version of  (i) and (ii),
is proved in Corollaries~\ref{cor:Z1} and \ref{cor:ii}.
The other parts of Theorem \ref{thm:1} concern $1$-parameter families, and their proof  mixes
geometric arguments with $J$-holomorphic curve theory.  The main idea is to show that for generic families $J_t$ one can find corresponding $1$-parameter families of embedded $A$-curves.  When $A \in \Ee$ (case (iii) of the theorem),  generic families of $A$-curves are embedded by positivity of intersections.  
For more general $A$, we formulate  
hypotheses that guarantee the existence of suitable embedded families in Proposition~\ref{prop:Z}.  
In \S\ref{ss:num} we then check  that these hypotheses hold in cases (iv) and (v). 

Sections \S\ref{s:geom} is essentially  independent of the rest of the paper. 
 In \S\ref{ss:geom} we explain how one might attempt a geometric construction of embedded $A$-curves. We give an extended  example (Example~\ref{ex:pole2}), and prove part 
 (ii) of Theorem \ref{thm:1} in Proposition~\ref{prop:ii}. 
The asymptotic result Theorem~\ref{thm:asympt} is explained in \S \ref{sec:asympt}.  The idea is that using Donaldson's construction of curves instead of Seiberg-Witten invariants and degenerations provides a much better control on the position of  the curve relative to $\Ss$. The smoothing process is then very elementary. However, one pays for this by having less control over the class that has the embedded representatives.  
Note also that the proof of Theorem~\ref{thm:asympt} depends on the existence of the 
 symplectic forms constructed in \S\ref{s:inflat}.

Finally \S\ref{s:inflat} deals with inflation, especially its $1$-parameter version that is also called \lq\lq deformation implies isotopy". This section provides explicit formulas for the inflation process along singular curves, and  gives complete proofs of Lemma~\ref{le:infl00} and Theorem~\ref{thm:inf}
 in the absolute and relative cases.  It relies on the results in \S\ref{ss:nodal} and \S\ref{ss:fam}, but is independent of the rest of \S\ref{s:proof} and of \S\ref{s:geom}.

\MS

\NI {\bf Acknowledgements.}\,  We warmly thank   Matthew Strom Borman,
 Tian-Jun Li  and Felix Schlenk  for
 very helpful comments on earlier drafts of this paper.

\section{Consequences of Taubes--Seiberg--Witten theory}\label{s:basic}

This section first
recalls various well known results on
$J$-holomorphic curve theory in dimension $4$, and then explains the modifications  necessary in the presence of a singular set $\Ss$.
Throughout, unless specific mention
 is made to the contrary,\footnote{
Occasionally we allow a curve to be disconnected, but it never has nodes unless the adjective \lq\lq nodal" or \lq\lq stable" is used.  }
by a {\bf
curve}
we mean the image of a smooth
map $u:\Si\to M$  where $\Si$ is a connected smooth Riemann surface.
 Thus   an {\bf immersed $J$-holomorphic
curve}
is the image of a smooth  $J$-holomorphic
immersion  $u:\Si\to M$. In particular,   all its double points have positive intersection number.  A curve is called {\bf simple}  (or {\bf somewhere injective}) if it is not multiply covered; cf. \cite[Chapter~2]{MS}.

\subsection{Review of $J$-holomorphic curve theory} \label{ss:SW} 

We begin this section by a brief review of Taubes' work
relating Seiberg--Witten theory to $J$-holomorphic curves
 in order  to explain the condition that $\Gr(A)\ne 0$. Here,   $\Gr(A)$ is Taubes' version of the Gromov invariant of $A$, that to a first approximation
counts embedded
$J$-holomorphic curves  in $(M,\om)$ through $\frac 12 d(A)$ generic points, where $d(A): = c_1(A) + A^2$ is the index of the appropriate Fredholm problem; cf. \cite{Tau1}.
Thus $\Gr(A)\ne 0$ implies both that $d(A)\ge 0$ and that $\om'(A)>0$
for all symplectic forms $\om'$ that can be joined to  $\om$ by a deformation (i.e. a path of possibly noncohomologous symplectic forms).
 For $4$-manifolds with $b^+_2=1$ (such as blow ups of rational and ruled manifolds), one shows that $\Gr(A)\ne 0$ by
using the wall crossing formulas in Kronheimer--Mrowka~\cite{KM} in the rational case and Li--Liu~\cite{LLw} in the ruled case.

When the intersection form on $H^2(M;\R)$
has type $(1,N)$, the cone $\Pp:= \{a\in H^2(M)\ | \ a^2>0\}$ has two components; let $\Pp^+$ be the component containing $[\om]$.  Then
we have the following useful fact.

\begin{fact}\label{f:*} Suppose that $b_2^+(M) = 1$.
If $a,b\in \ov{\Ppp^+}\less \{0\}$ then $a\cdot b\ge 0$ with equality only if $a^2=0$ and $b$ is a multiple of $a$.
 \end{fact}

Taubes' Gromov invariant $\Gr(A)$ in \cite{Tau1} counts holomorphic submanifolds and hence is somewhat different from the \lq\lq usual" invariant due to Ruan--Tian that counts (perturbed) $J$-holomorphic maps $u:(\Si,j)\to (M,J)$  with a connected domain of fixed topological type modulo reparametrization.
To explain the relation, we first make the following definition.

\begin{defn}\label{def:red}
A class $A\in H_2(M;\Z)$ is said to be {\bf reduced} if
 $A\cdot E\ge 0$ for all $E\in \Ee\less \{A\}$.
 \end{defn}

 For example, as noted at the beginning of \S\ref{ss:main}, every $E\in \Ee$ is reduced.
Now recall
that the adjunction formula for a somewhere embedded $J$-curve $u:(\Si,j)\to (M,J)$ in class $A$ with connected 
 smooth domain of genus $g_\Si$ states that 
\begin{equation}\label{adj} g_\Si\le  g(A): = 1+\tfrac 12(A^2-c_1(A)),
\end{equation}
 with equality exactly if $u$ is an embedding; cf \cite[Appendix~E]{MS}.   Using this,
 one can check 
 that the only reduced classes $A$ with $A^2<0$ and $d(A)\ge 0$ are those of the exceptional spheres 
 $A\in \Ee$ (see Remark \ref{rmk:nodaldim} (ii)).
 Taubes showed that if 
$A$ is reduced and
has 
$\Gr(A)\ne 0$
 then $A$ is represented by a holomorphic submanifold.   
Moreover, when $b_2^+(M)=1$ and $A^2>0$, it follows from Fact~\ref{f:*}  that this manifold is connected, while if
$A^2=0$ each component is a sphere or torus; cf.  Lemma~\ref{le:Jreg}.
In fact, except in the case of tori of zero self-intersection (where double covers affect the count in a very delicate way), the following holds.

\begin{fact}\label{f:A} {\it
 Assume that $A$ is reduced and,  if $g(A) =1$ and $A^2=0$, also  primitive.  
Then 
for generic $J$,
 the invariant $\Gr(A)$ simply counts (with appropriate signs) the number of possibly disconnected, 
embedded  $J$-holomorphic curves 
through  $\frac 12 d(A)$ generic points.  Moreover if  $b_2^+(M)=1$ this curve is connected with genus $g(A)$.
Thus $\Gr(A)$ equals the standard $J$-holomorphic curve invariant that counts connected
curves with genus $g(A)$ through $\frac 12 d(A)$ generic points.}
\end{fact}

  For example, $\Gr(E) = 1$ for all $E\in \Ee$.

Now let us consider a general, not necessarily reduced class, $A$ with $\Gr(A)\ne 0$. 
Then it is shown in McDuff~\cite[Proposition~3.1]{Mcmo}
that if we decompose $A$ as
\begin{equation}\label{eq:redu}
 A = A'+\sum_{E\in \Ee(A)} |A\cdot E|\, E,\quad \; \Ee(A) = \{E\ \bigl| \
 E\cdot A<0\},
\end{equation}
 then $E\cdot E'=0$ for $E,E'\in \Ee(A)$ and $A'$ is reduced with $\om(A')\ge0$, $d(A')\ge d(A)\ge 0$ and $A' \cdot E=0$ for all $E\in \Ee(A)$.   Further, for generic $J$ the class $A$ is represented by a main
 (possibly empty) embedded  component $C^{A'}$ in class $A'$, together with
 a finite number of disjoint curves $C^E$ each with multiplicity $|A\cdot E|$ in the classes $E\in \Ee(A)$.  This is proved by considering the structure of a
 $J$-holomorphic $A$-curve (where $J$ is generic) through $\frac 12 d'(A)$ points, where
\begin{equation}\label{eq:d'}
 d'(A): = c_1(A) + A^2 + \sum_{E\in \Ee}\bigl(|A\cdot E|^2 - |A\cdot E|\bigr).
\end{equation}
 It follows that $d'(A) = d(A')$.   Moreover, Li--Liu show in \cite{LL0}
 that the equivalence between Seiberg--Witten and Gromov invariants, previously established for reduced classes,  extends to show that the class $A$ has the same invariant as does its reduction $A'$.
 Thus:
 
 \begin{fact}\label{f:C} 
 Let $A'$ be the reduction of $A$, and assume that $A'$ is primitive if $g(A')=0, (A')^2=0$.
Then
 $\Gr(A)
   = \Gr(A')$
   counts the number of embedded $A'$-curves through $\frac 12 d'(A) = \frac 12 d(A')$ generic points.
In particular, if $A'=0$ then $d'(A) = d(A') = 0$.
\end{fact}

We next discuss conditions that imply $\Gr(A)\ne 0$.
The following is a sharper version of
Li--Liu~\cite[Proposition~4.3]{LL}. (Their result applies to more general manifolds.)

\begin{lemma}\label{le:SW}
\begin{enumerate} \item Let $(M,\om)$ be $S^2\times S^2$ or a blowup of $\C P^2$.  If $A\in H_2(M)$ satisfies $A^2\ge 0$, $\om(A)>0$,  and $d(A)\ge 0$, then $\Gr(A)\ne 0$.
\item Let $M$ be the $k$-point blowup of a ruled surface with base of  genus  $g(M)\ge 1$.  Then a
 sufficient (but not necessary) condition  for $\Gr(A)$ to be nonzero  is that
 $A\in \Ppp^+$ and $d(A)> g(M)+\frac k4$.
 \item Let $M$ be as in (ii) and $A\in H_2(M)$ be in the image of the Hurewicz map $\pi_2(M)\to H_2(M)$.  Then 
 $
d(A)\ge 0$ implies that $\Gr(A)\ne 0$.
\end{enumerate}
\end{lemma}
\begin{proof} We prove (i).  Since this can be proved by direct calculation when $M=S^2\times S^2$, we suppose that  $(M,\om)$ is obtained from
 the standard $\C P^2$ by blowing up  $N\ge 0$ points.
Let $E_i\in H_2(M), i=1,\dots,N,$ be the classes of the corresponding exceptional divisors.  Then the anticanonical class $K = -c_1(M)$ is standard, namely $K = -3L + \sum_{i=1}^N E_i$, where $L=[\C P^1]$. (As usual we identify $H_2(M;\Z)$ with $H^2(M;\Z)$ via Poincar\'e duality.)
Because $d(A) \ge 0$, it follows from the wall crossing formula
in \cite{KM}
that exactly one of $\Gr(A), \Gr(K-A)$ is nonzero.
Since  $A^2\ge 0$ and $\om(A)>0$,  the Poincar\'e dual of $A$
lies in $ \ov{\Ppp^+}$. Hence Fact~\ref{f:*}  implies that $\om'(A)>0$ for all forms $\om'$ obtained from $\om$ by deformation. On the other hand if
 $\om'(E_i)$ is sufficiently small for all $i$, $\om'(K-A)<0$. Therefore
 $K-A$ has no $J$-holomorphic representative for $\om'$-tame $J$, so that
 $\Gr(K-A)$ must be zero.  Hence $\Gr(A) \ne 0$. This proves   (i).

To prove (ii), we use \cite[Lemma~3.4]{LL} which states that $\Gr(A)\ne 0$ if $d(A)\ge 0$ and $2A-K\in \Ppp^+$.
Therefore (ii) will hold provided that $2A-K\in \Ppp^+$.  Suppose that $M=(S^2\times \Si_g)\# k\ov{\C P}^2$
is the $k$-fold blow up of the trivial bundle, where the exceptional divisors are $E_1,\dots,E_k$. Then
$$
K = -2[\Si_g]+2(g-1)  [S^2]  +\sum_{i=1}^k E_i
$$
 so that $K^2=-8(g-1)-k\le 0.$  Hence
the nontrivial ruled surface over $\Si_g$ also has $K^2=-4(g-1)$, since its  one point blowup is the same as
the one point blowup of the trivial bundle and blowing up reduces $K^2$ by $1$.
 Then
$$
(2A-K)^2 = 4A^2 - 4A\cdot K + K^2 =  4 d(A) - 4(g-1)-k>0,
$$
by our assumption.
Therefore either $2A-K\in \Ppp^+$ or $-(2A-K)\in \Ppp^+$.
But the displayed inequality also shows that  $A\cdot K \le  A^2 +\frac 14  K^2 \le A^2$ so that
 $A\cdot(2A-K) = 2A^2 - A\cdot K > 0$.  Hence, because $A\in \Ppp^+$, Fact~\ref{f:*}implies
 that  $2A-K\in \Ppp^+$ as required.

 To prove (iii), let us first consider the case when $(M,\om)$ is minimal.  Then $A = k[S^2]$, where $[S^2]$ is the class of the fiber.  Further  $k\ge 1$ since $d(A)=c_1(A)\ge 0$. Hence $\Gr(A): = \Gr(M,A)\ne 0$ by direct calculation. Note that this class takes values in $\Z\equiv\La^0H^1(M;\Z)$. We can now use the blow down formula of
 \cite[Lemma~2.8]{LL}.  This says that if $(X,\tau)$ is obtained from $(X',\tau')$ by blowing down the single exceptional class $E$, and if  $\Gr(X,B)$  takes values in
 $\Z\equiv \La^0H^1(X;\Z)$,  then
  for all $\ell\ge 0$
 $$
d(B-\ell E)\ge 0 \Longrightarrow  \Gr(X', B-\ell E) = \Gr(X,B)\in \Z.
 $$
 Note also that $d(B-\ell E) = (B-\ell E)^2 + c_1(B) + \ell = d(B) -\ell(\ell-1) \le d(B)$.
  Therefore if we start with a class in the $k$-fold blow up with $d(A)\ge 0$,  as we blow it down the degree $d(A)$ increases and we end up with a class $k[S^2], k>0$, in the underlying minimal ruled surface. Hence $\Gr(A)\ne 0$.
\end{proof}

\begin{cor}\label{cor:SW}  If $M$ satisfies any of the hypotheses in Lemma~\ref{le:SW}  and $A\in \Pp^+$, then there is an integer $q_0$  such that $\Gr(qA)\ne 0$ for all $q\ge q_0$.
\end{cor}
\begin{proof}  Since $(qA)^2>0$ grows quadratically with $q$ while $c_1(qA)$ grows linearly, the sequence $d(qA)$ is eventually increasing with limit infinity.   The result then follows from Lemma~\ref{le:SW}.
\end{proof}

The following recognition principle will be useful.
It is taken from \cite[Corollary~1.5]{MS1}, but here we explain some extra details in the proof.

\begin{lemma}\label{le:recog}\begin{itemize}\item[(i)]
Suppose that $(M^4,\om)$ admits a symplectically embedded submanifold $Z$ with $c_1(Z)>0$ that is not an exceptional sphere.  Then $(M,\om)$ is the blow up of a rational or ruled manifold.
\item[(ii)] The same conclusion holds if there is a $J$-holomorphic curve $u:(\Si,j)\to (M,J)$ 
in a class $B$ with $c_1(B)>0$, where $B\ne kE$ for some $E\in \Ee,k\ge 1$.
\end{itemize}
\end{lemma}
\begin{proof}  Since $0<c_1(Z) = 2-2g + Z^2$, where $g$ is the genus of the submanifold $Z$,
we must have $Z^2\ge 0$, since otherwise  $Z^2=-1$ and $g=0$ so that $Z$ is an exceptional sphere.  But when $Z^2\ge 0$ we can use the method of symplectic inflation from~\cite{Lal,Mcd}
to deform $\om$ to a symplectic form in class $[\om_\ka]: = [\om] + \ka \PD(Z)$ for any $\ka\ge 0$.  Therefore if $K$ is Poincar\'e dual to $-c_1(M)$, then $K\cdot Z <0$ so that
for large $\ka$ we have $\om_\ka(K)<0$.  But by Taubes' structure theorems in \cite{Tau}, this  is impossible when $b_2^+>1$. Thus $b_2^+=1$.  The rest of the proof of (i)  now follows the arguments given in~\cite{MS1}.  The crucial ingredient is Liu's result that a minimal manifold with $K^2<0$ is ruled.

This proves (i). To prove (ii),
note first that by replacing $u$ by its underlying simple curve we may assume that the map $u$ is somewhere injective.  Since this replaces the class $B$ by $B': = \frac 1k B$ for some $k>1$, we still have $c_1(B')>0, B'\notin\Ee$.  Then perturb the image of $u$ as in Proposition~\ref{prop:gp} below until it is symplectically embedded, and apply (i).
\end{proof}

We also recall from \cite{Tau} that for general $4$-dimensional symplectic manifolds, the classes with nonvanishing Gromov invariant are rigid:
\begin{fact}\label{fact:grdb} 
If $b_2^+>1$ and $\Gr(A)\neq 0$, then $d(A)=0$. 
\end{fact}

Finally we remind the reader of the standard theory of $J$-holomorphic curves as developed in \cite{MS}, for example.
  An almost complex structure $J$ is said to be {\bf regular} for a $J$-holomorphic map $u:(\Si,j)\to
(M,J)$ if the linearized Cauchy--Riemann operator $D_{u,J}$ is surjective.
 We will say that $J$ is {\bf semiregular} for $u$ if $\dim \coker D_{u,J}\le 1$.
 Here $(\Si_g,j_\Si)$ is a smooth connected Riemann surface,
 and when $g: = genus(\Si)>0$ we allow the complex structure
  $j_\Si$  on $\Si$ to vary, so that the tangent space $T_{j_\Si}\Tt$ at $j_\Si$ to Teichm\"uller space $\Tt$ is part of
  the domain of $D_{u,J}$; cf. \cite{Mcmo,Tau1}. 
Therefore, if $u$ is a somewhere injective curve in class $B$
 the (adjusted) Fredholm index of the problem  in dimension $2n=4$ is
 \begin{equation}\label{eq:Find}
\ind(D_{u,J}) = 2n(1-g) + 6(g-1) + 2c_1(B) = 2(g + c_1(B)-1).
\end{equation}
This is the virtual dimension of the quotient space of $J$-holomorphic maps modulo
 the action of the reparametrization group, where
we adjust by quotienting out by the reparametrization group (for genus $g_\Si=0,1$) and adding in the $6g-6$ dimensional tangent space to Teichm\"uller space when $g_\Si>1$.
Thus, if $J$ is regular,  the space  $\Mm_{g,k}(M,B,J)$,  of $J$-holomorphic maps $u: (\Si_g,j)\to (M,J)$ with $k$ marked points modulo reparametrizations and with $j$ varying in Teichm\"uller space,  is a manifold of dimension $\ind(D_{u,J})  + 2k$.
Hence the evaluation map 
\begin{equation}\label{eq:eval}
\Mm_{g,k}(M,B,J)\to M^k
\end{equation}
 can be locally surjective only if $\ind(D_{u,J})  + 2k\ge 4k$, i.e. $\frac 12 (\ind(D_{u,J})) \ge k$.

Now recall that the adjunction inequality~\eqref{adj} states that the genus $g(u)$ of the (connected) domain of any 
$J$-holomorphic curve in class $B$ satisfies
$g(u)\le g(B)$, where
the algebraic genus  $g(B)= 1 + \frac 12(B^2-c_1(B)$ is the genus of an embedded representative of $B$.  
Therefore, \eqref{eq:Find} gives
$$
\ind D_{u,J}= 2\bigl(c_1(B)+g(u)-1\bigr)=c_1(B)+B^2+2(g(u)-g(B)).
$$ 
In other words 
\begin{equation}\label{eq:ind<d}
\ind D_{u,J}= d(B)+2(g(u)-g(B))\leq d(B).
\end{equation}

\begin{rmk}\label{rmk:nodaldim}\rm (i)
The above inequality \eqref{eq:ind<d} implies that when $J$ is regular for all $B$-curves the evaluation map 
$\Mm_{g,k}(M,B,J)\to M^k$ can be surjective only if $\frac 12 d(B)\ge k$.  Informally, we may say that a connected $B$-curve 
can go through at most $\frac 12 d(B)$ generic points of $M$.   Note that  nodal 
{\it regular} 
curves do worse. If $\Si^B$ is a $J$-holomorphic nodal curve in class $B$ with components in classes $B_j, j=1,\dots,n$ then 
positivity of intersections implies that $B_i\cdot B_j\ge 0$ for all $i\ne j$ so that $\sum d(B_j)\le d(B)$,
 with strict inequality if any $B_i\cdot B_j>0$.  Hence if all 
components of $\Si^B$ are regular and some $B_i\cdot B_j>0$
(which always happens when $\Si^B$ arises as a Gromov limit of connected curves),
 then such a nodal curve goes through at most $\frac12 \sum_j d(B_j)< \frac 12 d(B)$ points. 
However, if  some of the components of $\Si^B$ are not regular (e.g. they lie in the singular set $\Ss$, 
or  they are multiply covered exceptional spheres), their Taubes index might be negative, so
 others may have larger index, 
 and could go through more points. 
 The arguments that follow  show how to deal with this problem in certain special cases.\MS

\NI (ii) If $A^2<0$, the condition $d(A)\geq 0$, combined with the formula $$
d(A)=2(A^2-g(A)+1),
$$
 shows that $A^2=-1$, 
$g(A)=0$.  Hence $g(u)=g(A)=0$, so $u$ is an embedded exceptional sphere. 
\end{rmk}

\subsection{The case $J\in \Jj(\Ss)$} \label{ss:SWSS}

We now suppose that $J$ belongs to the set $\Jj(\Ss)$ of Definition~\ref{def:sing},
where this is given the direct limit topology.
When we consider $J$-holomorphic representatives for
a reduced class $A$ for such $J$, the situation is rather different
from before
 since the curves in $\Ss$ are not regular.
 Thus $A$ could decompose as $A = \sum_i\ell_i S_i + A'$ where $\ell_i\ge 0$, and we need to consider generic 
 representations of the class $A'$.  But  $A'$ need not be
 reduced, and hence could have a disconnected
 representative as above with some  multiply covered exceptional spheres.
We will consider two subsets of $\Jj(\Ss)$, first a set (defined carefully below) of regular $J$, that we call
 $\Jj_{reg}(\Ss)$, and secondly a larger path connected set $\Jj_{semi}(\Ss)$ whose elements retain some of
  the good properties   of regular $J$. 
    Specially important will be certain special paths in $\Jj_{semi}$ called regular homotopies.

  \begin{defn}\label{def:Jreg}  If $\ov\Nn$ is a closed fibered neighborhood  of $\Ss$, the
  space $\Jj_{reg}(\Ss, \ov \Nn,\om,\ka)$ of {\bf regular $(\Ss,\ov\Nn)$-adapted $J$} is the set  of  almost complex structures
   $J\in \Jj(\Ss)$ satisfying the following conditions:
    \begin{itemize} \item [(i)] $J$ is $\Ss$-adapted on $\ov \Nn$;
    \item [(ii)] $J$ is regular for all somewhere injective elements $ u: (\Si,j_\Si)\to (M,J) $  in class $B$
 with $\om(B)\le \ka $ and $\im u \cap (M\less \ov \Nn)\ne \emptyset$. 

      \end{itemize}
The space $\Jj_{semi}(\Ss,\ov \Nn,\om,\ka)$ of {\bf semiregular $\Ss$-adapted $J$} 
    consists of all $J\in \Jj(\Ss)$  that are semiregular for all  maps $u$ satisfying the above conditions.  
    We then define
    $$
     \Jj_{reg}(\Ss,\om,\ka): = \bigcup_{\ov \Nn}     \Jj_{reg}(\Ss,\ov \Nn,\om,\ka),\qquad 
      \Jj_{semi}(\Ss,\om,\ka): = \bigcup_{\ov \Nn}     \Jj_{semi}(\Ss,\ov \Nn,\om,\ka),
      $$
      and give these spaces the direct limit topology.
\end{defn}
     
\begin{rmk}\label{rmk:d}\rm (i)  In the case of spheres there is a close 
connection between the value of  the Chern class $c_1(B)$ and the (semi)regularity of a somewhere injective
$J$-holomorphic sphere
$u:(S^2,j)\to (M^4,J)$ in class $B$. Indeed, if $J\in \Jj_{semi}(\Ss,\om,\ka)$ for some $\ka\ge \om(B)$
and $B$ is represented by a somewhere injective curve that meets $M\less \ov\Nn$, then
$c_1(B)>0$ because 
$\ind D_{u,J} = 2c_1(B)-2$.   Conversely, if $u$ is immersed, then the condition 
$c_1(B)>0$ implies  the surjectivity of $D_{u,J}$ by automatic regularity \cite{holisi}.
\MS

\NI (ii) If $\Aa= \{A_1,\dots,A_k\}$ is a finite set of  reduced classes $A_j$, 
  we define the  space $\Jj_{\nf{reg}{semi}}(\Ss,\om,\Aa):=\Jj_\nf{reg}{semi}(\Ss,\om,\ka(\Aa))$ of  
  almost complex structures, where $\ka(\Aa)=\max_j\om(A_j)$.
   In practice, these complex structures are (semi)regular at each component not  in $\Ss$
    of the stable maps that represent the $A_j$. 
\end{rmk}

 \begin{lemma}\label{le:reg}  
 The subset  $ \Jj_{reg}(\Ss, \ka)$
 of $\Jj(\Ss)$ is residual in the sense of Baire. Further,
 $ \Jj_{reg}(\Ss, \ka)\subset  \Jj_{semi}(\Ss,\ka)$.
 \end{lemma}
 \begin{proof}  Let $\Jj(\Ss, \ov \Nn,\ka)$ denote the subset of $\Ss$-adapted $J$ satisfying Definition~\ref{def:Jreg}~(i) for the given $\ov \Nn$.   
Because  $\Jj_{reg}(\Ss, \ka)$  is a (countable) direct limit,
it suffices to check that $\Jj_{reg}(\Ss, \ov \Nn,\ka)$ is residual in 
$\Jj(\Ss, \ov \Nn,\ka)$ for each $\ov \Nn$.    When the domain $\Si$ of $u$ has genus zero this follows immediately from  standard theory as developed  in \cite[Chapter~3.2]{MS}, since we can vary $J$ freely somewhere on $\im u$.
The argument applies equally in the higher genus case.   One main technical ingredient is the
 version of the Riemann--Roch theorem in
 \cite[Theorem~C.1.10]{MS}. Since  this theorem is stated for  arbitrary genus,
 one can easily adapt the above proof to higher genus curves as in \cite{Tau1,Mcmo}.
This proves the first statement. The rest of (i)
is then immediate since the elements in $\Jj_{semi}(\Ss,\ka)$ 
satisfy fewer conditions than those in $\Jj_{reg}(\Ss,\ka)$.
 \end{proof}

\begin{lemma}\label{le:Jreg} 
Let
 $J\in  \Jj_{semi}(\Ss,\ov\Nn, \Aa)$.
  The following statements hold for
 somewhere injective  $J$-holomorphic curves $u$ in a class $B$ with $\om(B)\le \ka(\Aa)$.
 \begin{itemize}\item [(i)]   If $B\ne \sum\ell_iS_i$ with $\ell_i\ge 0$, then $\im u\cap (M\less \ov \Nn)\ne \emptyset$.  
 \item [(ii)] If $\im u\cap (M\less \ov \Nn)\ne \emptyset$  
 then $d(B)\ge 0$.  Moreover, $B^2\ge 0$  unless
 $B\in \Ee$,  and if $B^2=0$ then $B$ is represented by an embedded $J$-holomorphic sphere or torus.
 \end{itemize}
\end{lemma}
\begin{proof} 
Let $J\in \Jj(\Ss)$ be $\Ss$-adapted on some fibered neighborhood $\Nn(\Ss)$. 
If $u:(\Si,j)\to (\ov\Nn,J)$
is $J$-holomorphic, then $B=\sum\ell_iS_i$ for some $\ell_i$,
 because  there is a projection $\ov \Nn\to \Ss$. Moreover $\ell_i\ge 0$  because we can choose this projection to be $J$-holomorphic over some nonempty open subset of each curve $C^{S_i}$ in $\Ss$.  
This proves (i).

To prove (ii),
notice that since
the index of a somewhere injective $J$-holomorphic curve 
 with domain of genus $g$ is even and  $\dim \coker (D_{u,J})\le 1$ when $J\in \Jj_{semi}(\Ss,\Aa)$, we must have
$\ind(D_{u,J})\ge 0$. Hence $d(B)\geq \ind(D_{u,J})\ge 0$ by equation \eqref{eq:ind<d}.
  Further,  the only simple curves in a class $B$ with 
   $B^2<0$ and $d(B) \ge 0$ are embedded exceptional spheres (Remark \ref{rmk:nodaldim} (ii)).  
 Similarly,  if 
 $B^2=0$ we again have equality
 in the adjunction formula , so that the curve is embedded and $g(B)=0$ or $1$, as claimed.
\end{proof}

\begin{rmk}\label{rmk:reg1}\rm  A path $J_t\in \Jj(\Ss), t\in [0,1],$ is called an {\bf $(\Ss,\ov\Nn)$-regular homotopy}
if the  derivative $\p_t J_t$ covers the cokernel of $D_{u, J_t}$  for every map $ u: (\Si,j_\Si)\to (M,J) $ 
that satisfies condition (ii) in Definition~\ref{def:Jreg}.  
Thus  $(J_t)$ is a path in $\Jj_{semi}(\Ss,\ov\Nn,\om,\ka)$ with the special property that for each $t$ all the relevant  cokernels
are covered by the (restriction of the) single element $\p_t J_t$. The proof of 
\cite[Theorem~3.1.7]{MS}   shows that
any two elements $J_0, J_1\in \Jj_{reg}(\Ss, \ov\Nn,\om,\ka)$ may be joined by
a regular homotopy of this kind.  

Let us denote by $\Mm_{g,k}(M\less \ov\Nn; B,J_t)$ the moduli space of all $k$-pointed curves as in \eqref{eq:eval} whose image meets $M\less \ov\Nn$.  Then \cite[Theorem~3.1.7]{MS}  also shows that,
 for each $B$ with $\om(B)\le \ka$,
the moduli space $\bigcup_{t\in [0,1]} \Mm_{g,k}(M\less \ov\Nn;B,J_t)$ is a smooth manifold of the \lq\lq correct" dimension
$\ind D_{u,J} + 2k+1$
with boundary at $t=0,1$.  Hence the corresponding evaluation map goes through at most 
$\frac 12 d(B)$ generic points in $M\less \ov \Nn$; cf. Remark~\ref{rmk:nodaldim}.  Note also that if $B\ne \sum_im_iS_i, m_i\ge 0,$
then every $B$-curve meets $M\less \ov \Nn$ by Lemma~\ref{le:Jreg}~(ii).  Therefore, in this case
$ \Mm_{g,k}(M\less \ov\Nn;B,J_t)= \Mm_{g,k}(M;B,J_t)$.
\end{rmk}

\section{The proof of Theorem~\ref{thm:1}}\label{s:proof}
We first explain 
the structure of nodal representatives of $A$, and then in Proposition~\ref{prop:gp}  show how to 
build embedded curves from components in classes $B$ with $B^2\ge 0$.    As we see in Corollary~\ref{cor:gp0}
and Proposition~\ref{prop:Ee}, these arguments suffice to prove Theorem~\ref{thm:1} in cases (i) and (iii).
\S\ref{ss:fam} explains how to construct $1$-parameter families of embedded curves, while \S\ref{ss:num}   
proves Theorem~\ref{thm:1} in cases (iv) and (v).
 
  \subsection{The structure of nodal curves}\label{ss:nodal}

Throughout this section we assume that the class $A$ is $\Ss$-good
in the sense of
of Definition~\ref{def:*S}.
For such $A$,  as explained in \S\ref{ss:SW}
there is for each generic $\om$-tame $J$ and
each sufficiently generic set of $\frac 12 d(A)$ points in $M$
an embedded $J$-holomorphic curve $u:(\Si,j)\to (M,J)$  of genus 
$g(A): = 1+\frac 12(A^2-c_1(A))$ through these points.
Hence by Gromov compactness, for every  $\om$-tame $J$ and every set of $\frac 12 d(A)$ points, 
there is a connected but possibly nodal representative of the class $A$ through these points that is the limit of these embedded curves.
We denote such nodal curves  as $\Si^A$, reserving the notation $C^A$ for a (smooth, often immersed) curve.  This section explains the structure of these nodal curves. Recall from Definition~\ref{def:sing} that $\Jj(\Ss,\ov\Nn)$ 
consists of $\om$-tame $J$ that are fibered on the \nbd\,$\ov\Nn$ of $\Ss$.

\begin{lemma}\label{le:SiA1}  For each $J\in \Jj(\Ss, \ov\Nn)$ and 
$\Ss$-good class $A$, there is a 
connected $J$-holomorphic nodal curve $\Si^A$ in class $A$  whose components are either
 multiple covers of the components of $\Ss$ or lie in classes $B_j \ne \sum_i m_i S_i, m_i\ge 0$.  
 The homology classes of these components 
provide a decomposition
 \begin{equation}\label{eq:Adecomp0}
 A = \sum_{i=1}^s\ell_i S_i + \sum_{j=1}^k n_j  B_j ,
\end{equation}
satisfying
 \begin{itemize}
 \item[(i)] $\ell_i\ge 0$ and $n_j >0$ for all $j$;
\item[(ii)] $B_j \cdot S_i\ge 0$ for all $i,j$; 
\item[(iii)] each class $B_j $ may be represented by a 
connected
simple
 $J$-holomorphic curve $C^{B_j}$ that intersects $M\less \ov\Nn$.
\end{itemize}
Further  every $J$-holomorphic nodal curve  $\Si^A$ 
that is the Gromov limit of
 embedded $J_n$-holomorphic $A$-curves for some convergent sequence $J_n$ 
 has this structure.
\end{lemma}
\begin{proof}  Let $\Si^A$ be any $J$-holomorphic nodal curve. As explained above, these exist because $\Gr(A)\ne 0$.
Then 
since we may replace every component
in some class  $\sum_j m_i S_i, m_i\ge 0,$ by a union of  copies of the $C^{S_i}$, we can suppose that no $B_j$
has this form.  Therefore $A$ does decompose as in \eqref{eq:Adecomp0}, and (i) and (ii) hold  
by positivity of intersections.    To prove  (iii), note first that we may 
 take $B_j$ to be the class of
a simple curve underlying a possibly multiply covered component of $\Si$.   
The curve $C^{B_j}$ must intersect  $M\less \ov\Nn$ by Lemma~\ref{le:Jreg}~(i).
 \end{proof}

The following sharpening of this result is useful in proving  Theorem \ref{thm:1}.
Order  the classes  $B_j $ (assumed distinct) so that $B_j \in \Ee$ for $j\le p$ and $B_j \notin\Ee$ otherwise, and
write $E_j : = B_j $ for $1\le j\le p$, and  $B : = \sum_{j>p} n_j B_j $. We then have
\begin{equation}\label{eq:Abir}
A \ =\ \sum_i\ell_i S_i + \sum_{j=1}^p m_j  E_j  + \sum_{j>p}
n_j B_j \ =\ \sum_i\ell_i S_i + \sum_{j=1}^p m_j  E_j  + B .
\end{equation}
where $B \cdot (A-B )>0$ if $B\ne 0$ because  $\Sigma^A$ is connected.

\begin{lemma}\label{le:Abir}  Suppose that 
 $J\in \Jj_{semi}(\Ss,A)$ and that the $J$-holomorphic nodal curve $\Si^A$ is the Gromov 
 limit of embedded curves.  
  Then
 the components $B_j, j>p$, in  its decomposition~\eqref{eq:Abir}
 also satisfy
\begin{itemize}
\item $d(B) \ge \sum_{j>p} d(B_j)\ge 0$;
\item  $B_j^2\ge 0$ for all $j>p$.
\end{itemize}
Moreover,  if $B_j$ is represented by a $J$-sphere, we have
$\Gr(B_j)\ne 0$.  (This case occurs only if  $M$ is the blow up of a rational or ruled manifold.) 
\end{lemma}
\begin{proof}  Apply Lemma~\ref{le:SiA1} to $\Si^A$.  
By Lemma~\ref{le:SiA1}~(iii), we can apply Lemma~\ref{le:Jreg}~(ii) to curves in class $B_j$ and hence establish the first claim.
Since $d(B_j)\geq 0$, Remark \ref{rmk:nodaldim} (ii) shows that either $B_j^2\ge 0$ or $B_j$ is represented by a $J$-holomorphic $-1$-sphere. The latter is ruled out by definition, so $B_j^2$ is indeed nonnegative $\forall j> p$. 
 When $B_j$ is represented by a $J$-sphere $u: (S^2,j)\to (M,J)$, 
then we saw in Remark~\ref{rmk:d}~(i) that $c_1(B_j)>0$.  
Therefore, because $B_j^2\ge0$ we also have $d(B_j)> 0$.
Therefore Lemma~\ref{le:recog}~(ii) implies  that
 $M$ is the blow up of a rational or ruled manifold.  Finally
because the class $B_j $ is the
 $J$-holomorphic  image of a sphere, we conclude from Lemma~\ref{le:SW} parts (i) and (iii)
 that  $\Gr(B_j )\ne 0$. 
\end{proof}

These lemmas give enough preparation for the proof of 
part (iii) of Theorem~\ref{thm:1} (the case $A\in \Ee)$.  
We next prove a  general position result that allows us to \lq\lq clean up"  a nodal
 representation of the class $A$.  
  The result when 
 $\Ss=\emptyset$ is well known.  
 Besides being the key to the 
handling of the components in $\Ss\less \Ss_{sing}$, 
this lemma  will be very useful when discussing inflation in \S\ref{s:inflat}.
Note  that  
in distinction to the decomposition  $B=\sum n_jB_j$ considered above
 where by definition $B_j\ne S_i$ for any $i,j$,
 we now allow $T_j = S_i$ in some cases.

\begin{prop}\label{prop:gp} Let $T=\sum_{j=1}^N n_j T_j\in H_2(M)$ be such that
\begin{itemize}
 \item[(i)] $T_j\ne T_k$ for each $j\ne k$, and $n_j \ge 1 $;
 \item[(ii)]  for some $J_0\in \Jj(\Ss)$, each $T_j$
 can be represented by a simple connected (non-nodal) $J_0$-holomorphic curve 
 $C^{T_j}$;
 \item[(iii)] 
 $T_j^2\ge 0$ unless $C^{T_j}$ is an exceptional sphere;
 \item[(iv)]  $T\cdot S_i\ge 0$ for all $i$ and $T\cdot T_j\ge 0$ for all $j$; 
 further,  $T_j\cdot S_i\ge 0$ for all $i,j$
unless $T_j = S_i$ where  $C^{S_i}$ is an exceptional sphere.
 \end{itemize}
Then,   $T$ can also be represented
by a (possibly disconnected) embedded  curve that is orthogonal to
$\Ss$ and  $J$-holomorphic for some $J\in \Jj(\Ss)$.
\end{prop}
\begin{proof} {\bf Case 1:} {\it  We assume
  $N=n_1 = 1$.}\
If $T=S_i$ for some $i$, then there is the required  embedded representative, namely $C^{S_i}$.
Therefore, assume $T\ne S_i$ for any $i$.
By hypothesis
there is a connected simple $J_0$-holomorphic curve $C^T$, and our task is to resolve its singularities to make it embedded.   By general theory (see for example~\cite[Appendix~E]{MS})
 $C^T$ has at most a finite number of singular points $q_i=u(z_i)$. Suppose first that none lie on $\Ss$.    
 At each of these it is possible to perturb $C^T$ locally to an immersed $J_0$-holomorphic curve
 by \cite[Theorem~4.1.1]{Mcsing},
 and then patch this  new piece of curve to the rest of $C^T$ by the technique of
 \cite[Lemma~4.3]{Mclb},
  to obtain a positively immersed symplectic curve $C'$.  
    The curve $C'$ is
  $J_0$-holomorphic except close to $C^T\cap {\it Shell}$, where ${\it Shell}$ is the union of
 spherical shells ${\it Shell}(q): = T_{r_1}(q)\less T_{r_2}(q)$ centered at the finite number of singular points $q$. Thus we can make it $J$-holomorphic for some $J$ near $J_0$ that equals $J_0$ away from  $C'\cap Shell$.
  Hence even if some singular point $q$ is in some $C^{S_i}$
  we can assume $J\in \Jj(\Ss)$.

Then $C'$ is immersed, and can be homotoped
(keeping it symplectic) so that it has at
most transverse double points that are disjoint from its intersections with the 
curves $C^{S_i}$ in $\Ss$. Then we  deform $C'$ so that it is vertical near its intersections $p$ with
each $C^{S_i}$, in the sense that it coincides with the  fiber of the normal bundle 
to $\Ss$ at $p$. 
(A parametric version of this maneuver is carried out in more detail in  Lemma~\ref{le:gp1} below).
Then $C'$ meets each component $C^{S_i}$ of $\Ss$ orthogonally in distinct points.
 Moreover, by resolving all its double points (which lie away from $C^{S_i}$),
 we can assume that $C'$ is embedded and still $J$-holomorphic for some $J\in \Jj(\Ss)$.
 This completes the proof when  $N=n_1=1$.   Notice also that $C'$ is connected since we assumed that the initial curve $C^{T}$ is connected.
\MS

\NI {\bf Case 2:} {\it We assume $T=nT_0$ where $n>1$ and $T_0\ne S_i$ for any $i$.}
By the above we can suppose that $C^{T_0}$ is embedded,
 orthogonal to $\Ss$ and $J$-holomorphic for some $J\in \Jj(\Ss)$.
   Then for suitable $J\in \Jj(\Ss)$ a neighborhood  $\Nn(C^{T_0}, J)$ 
of $C^{T_0}$ can be identified with a neighborhood of the zero section in a 
holomorphic line bundle over $C^{T_0}$ with nonnegative Chern class.
(Since $n>0$ condition (ii) implies that $(T_0)^2\ge 0$.)
Moreover, since the condition $J\in \Jj(\Ss)$ only affects the complex structure on 
$\Nn(C^{T_0})$  near a finite set of points, we may choose $J$ so that
this bundle has nonzero holomorphic sections.
Hence we may represent the class $n T_0$ by the union of $n$ generic 
$J$-holomorphic sections of this bundle that intersect transversally.  
If $T_0^2>0$ each pair of these sections intersect, and by choosing generic 
sections we can assume that the intersection points do not lie on $\Ss$.  
Hence after resolving these intersections as before, we get an embedded (possibly disconnected) 
representative of $nT_0$ that we finally perturb to be orthogonal to $\Ss$.
\MS

\NI {\bf Case 3:} {\it We assume $T=nT_0$ where $n>1$ and $T_0= S_i$ for some $i$.}
This is much as Case 2: we just need to pick $J\in \Jj(\Ss)$ so that  the normal bundle to $C^{T_0}=C^{S_i}$ has holomorphic sections that intersect the zero set transversally in a finite number of points.  This is possible because by condition (iv) we have 
$T\cdot S_i = n(T_0)^2\ge 0$.
\MS

\NI {\bf Case 4:} {\it We assume $N>1$ and $T_j^2\ge 0$ for all $j$.}
We first resolve all singularities, so that each simple curve
$C^{T_j}$ is embedded and meets
all the other curves $C^{T_k}$  and $C^{S_i}$ transversally and positively in double points.
Because  $T_j^2 \ge 0$, even  if  $T_j = S_i$   
 we may replace  $C^{T_j}$  by a suitable section of its normal bundle 
 that is transverse to $C^{S_i}$.
Next, we perturb all double points to be orthogonal.  
Since $(T_j)^2\ge 0$ by assumption, we may represent every class $n_jT_j$ by 
embedded curves as in Cases 2 and 3 above.  
Finally, we patch all double points to get an embedded curve in class $T$.
\MS

\NI {\bf Case 5:} {\it  The general case.}  
Because $T\cdot  T_j\ge 0$ for all $j$, each exceptional 
class $T_j$  must intersect some other component in $T$.  
 If two different  exceptional spheres $C^{T_k}, C^{T_\ell}$ intersect,
 then we may  form a symplectically embedded  curve $C'$ 
 that is transverse to $\Ss$ 
 by patching together two meromorphic and nonvanishing sections of their normal bundles each with a single pole at the intersection point.   Then by perturbing $C'$ further we can suppose that it is 
 $J$-holomorphic for some $J\in \Jj(\Ss)$.
 Therefore we can replace these two components $T_k, T_\ell$ of $T$ with the single component $T': = T_k+T_\ell$.  Further,
 the decomposition 
 $T =  T' + \sum n_j'T_j$, where $n_j '= n_j-1$ for $j= k,\ell$ and $= n_j$ otherwise, satisfies all the conditions (i) through (iv).  In particular, by construction $T_j\cdot T' =0 = T_k\cdot T'$.
 Because the meromorphic sections do not vanish, this procedure works equally well if one or both
 spheres  $C^{T_k}, C^{T_\ell}$  are in $\Ss$.  It also works
 if  an exceptional sphere $C^{T_k}$ intersects some nonnegative component of $T$.
 Therefore, after a finite number of steps of this kind, we arrive at a decomposition
 $T = \sum n_j' T_j'$  with no exceptional spheres, and hence the conclusion follows by Case 4.
\end{proof}

\begin{cor}\label{cor:gp0}  Part (i) of Theorem~\ref{thm:1} holds.
\end{cor}
\begin{proof}   
 Suppose that $\Ss = \Ss_{reg}\cup \Ss_{pos}$, let $J\in \Jj_{semi}(\Ss)$, and
choose a $J$-holomorphic nodal representative $\Si^A$ of $A$  as in Lemma~\ref{le:Abir}.
Then write $A = \sum n_j T_j$ where $T_j$ is one of the classes 
$S_i, E_j, B_j$ occurring in \eqref{eq:Abir}.  By assumption, the classes $S_i$ either have 
$(S_i)^2\ge 0$ or 
are represented by an embedded curve $C^{S_i}$ that is Fredholm regular and hence
must be an exceptional sphere by Remark~\ref{rmk:nodaldim}~(ii). 
Therefore this decomposition of $A$ satisfies all the conditions (i) through (iv) in Proposition~\ref{prop:gp}.
In particular, because $A$ is reduced we must have $A\cdot E_j\ge 0$, and $A\cdot S_i\ge 0$ 
because $A$ is $\Ss$-good. 
 Therefore the result follows from Proposition~\ref{prop:gp}.
\end{proof}

For the next result, denote by  $\Ii_{neg}$, respectively $\Ii_{nonneg}$,  the classes with $(S_i)^2<0$, respectively
$(S_i)^2\ge 0$.   Recall from Remark~\ref{rmk:sing} 
that the elements in $\Ii_{neg}$ are either represented by exceptional spheres or are in $\Ii_{sing}$.

\begin{lemma}\label{le:gp2}   Suppose that $A$ is $\Ss$-good.
Then 
we may write 
\begin{equation}\label{eq:Abir1}
A \ =\  \sum_{i\in \Ii_{neg}}\ell_i S_i + \sum_{k=1}^q m_k  E_k  +  B,\quad \ell_i\ge 0,\;\;  m_k>0,
\end{equation}
where
\begin{itemize}
\item[(i)] if $\ell_i>0$ and $C^{S_i}$ is an exceptional sphere, then $S_i\cdot B =0$;
\item[(ii)] 
each $E_k$ for $k\le q$ satisfies $E_k\cdot E_j=0, j\ne k$, $E_k\cdot S_j\ge 0$ for $1\le j\le s$
and $E_k\cdot B=0$;
\item[(iii)]   $B$
has an embedded representative $C^{B}$ that intersects $M\less \Ss$ and is $J$-holomorphic for some $J\in \Jj(\Ss)$;
\item[(iv)]  if all $S_i$ with $S_i^2\ge 0$ are regular, then $d(B)\ge 0$.
\end{itemize} 
\end{lemma}
\begin{proof}   
Consider a decomposition of $A=\sum_i\ell_i S_i + \sum_{j=1}^p m_j  E_j  + B $  
as in \eqref{eq:Abir} given by a $J$-holomorphic nodal curve where $J\in \Jj_{semi}(\Ss)$.
As in the proof of Proposition~\ref{prop:gp} we may incorporate all nonnegative components $\ell_iS_i$ 
into $B$.\footnote
{
Since these components need not be Fredholm regular and could have $d(S_i)<0$, we may lose control of $d(B)$ at this step.}
   If $E_j\cdot E_k>0$, 
then, as in the proof of Case 5 of 
Proposition~\ref{prop:gp},  we may reduce each of $m_j, m_k$ by $1$ and 
add a component in class $E_j+E_k$ to  $B$.  
Similarly, if $E_j\cdot B_k>0$, or if 
 $E_j\cdot S_i>0$ or $B_j\cdot S_i>0$ for some $i$ for which $C^{S_i}$ is an exceptional sphere, we may 
incorporate one copy of the exceptional class $S_i$ or $E_j$ into the $B_j$. 
Repeating this process, we arrive at a situation in which (i) and (ii) hold, and $B$  (if nonzero) has an embedded representative  $C^B$ that intersects $M\less \Ss$ and is $J$-holomorphic for suitable $J\in \Jj(\Ss)$.
 
To prove (iv), notice that if there are no irregular nonnegative components,  $d(B)$ cannot decrease as we incorporate the various components $C^{S_i}$, and $C^{E_j}$ into the $B$-curve.  Because we begin with $ d(B)\ge 0$  by Lemma~\ref{le:Abir}, 
this proves (iv).
\end{proof}

We end this section by proving case (iii) of Theorem~\ref{thm:1}.  

\begin{prop}\label{prop:Ee} Theorem~\ref{thm:1}  holds when $A\in \Ee$. 
Moreover, we may choose $\Jj_{emb}(\Ss,A)\supset \Jj_{semi}(\Ss,A)$.
\end{prop}
\begin{proof}
We first  show that when $J\in   \Jj_{semi}(\Ss, A)$
 each  $A\in \Ee$ has an embedded $J$-holomorphic representative.
Suppose, to the contrary, that this does not hold for some 
$\Ss$-good $A\in \Ee$  and some $J\in   \Jj_{semi}(\Ss, A)$. 
Consider the $J$-holomorphic  nodal representative $\Si^{A}$  
with decomposition  $$
A = \sum_i\ell_i S_i +\sum_{j=1}^p m_jE_j+ \sum_{j>p } n_jB_j 
$$ 
as in \eqref{eq:Abir}.  Since $\Si^A$ is the Gromov limit of spheres, each component of $\Si^A$ is represented by a sphere.  
 If there is just one component, this  must be somewhere injective since the class $A$ is primitive, and hence by the
 adjunction formula  \eqref{adj} must be embedded. 
 Therefore, we can assume that $\Si^{A}$  has several components.   Because $A\cdot A = -1$
 and $A\cdot S_i\ge 0$ for all $i$, the class $A$ must have negative intersection with  one of the $E_j$ or $B_j$.
 But because this decomposition is nontrivial, $\om(E_j)<\om(A)$ for each $j$, so that 
 $A\ne E_j$. Hence, because $A, E_j\in \Ee$, we must have  
 $A\cdot E_j\ge 0$ for all $j$.
 Therefore, there is $j>p$ such that  $A\cdot B_j<0$. 
   Next, notice that 
 $\Gr(B_j)\ne 0$ by 
  Lemma~\ref{le:Abir}.    
Therefore, by Fact~\ref{f:A}  for generic $J'\in \Jj(M)$ the class  $A$ has an embedded   $J'$-holomorphic
representative while, by the discussion after \eqref{eq:redu}, $B_j$ (which need not be reduced)
can be represented by an embedded curve in some class $B_j'$  together with possibly  multiply covered exceptional spheres in classes $E_\al'$.
But $E_\al'\ne A$ since $\om (E_\al')\le \om (B_j) < \om (A)$.  Hence $A\cdot E_\al'\ge 0$, and also $A\cdot B_j'\ge 0$.
Therefore $A\cdot  B_j\ge 0$, which contradicts the choice of $B_j$.
We conclude that the class $A$ must have 
an embedded $J$-holomorphic representative for each
$J\in  \Jj_{semi}(\Ss,A)$.   Further,  $A$ can have no other nodal $J$-holomorphic representative,
since if it did $A$ would have nonnegative intersection with each of its components, and hence with $A$ itself, 
which is impossible because $A\in \Ee$.

Next define $\Jj_{emb}(\Ss,A)$ to be the set of $J\in \Jj(\Ss)$ for which $A$ has an embedded representative. 
This set is residual 
in $\Jj(\Ss)$,
 because it contains $ \Jj_{semi}(\Ss, A)$, which is residual by Lemma~\ref{le:reg}.  
Further, it is open since embedded curves in class $A$ are regular by automatic regularity (cf. Remark~\ref{rmk:d})
and hence deform to nearby embedded curves when $J$ deforms.
It remains to check that $\Jj_{emb}(\Ss,A)$ is path connected. 
  But this holds because
 any two elements $J_0, J_1\in \Jj_{emb}(\Ss,A)$ can be
 slightly perturbed  to  $J_0', J_1'\in\Jj_{emb}(\Ss,A)\cap \Jj_{reg}(\Ss,A)$, and then by 
Remark~\ref{rmk:reg1}, joined by a regular homotopy in $\Jj_{semi}(\Ss,A)\subset \Jj_{emb}(\Ss,A)$.  
 \end{proof}

\begin{rmk}\rm Of course, classes $E\in \Ee$ do degenerate, for example as $(E-E') + E'$ where $E'\in \Ee$.  But such degenerations  (a) happen for $J$ in a set of codimension at least $2$, and (b) have the property that the intersection of $E$ with the class of the nonregular component(s) (in this case $E-E'$) is {\it negative}.  The  argument above shows the presence of nonregular components in classes $S_i$ with $E\cdot S_i\ge 0$ does not affect the situation.
\end{rmk}

  \subsection{One parameter families}\label{ss:fam}
We begin with a useful geometric result.

 \begin{lemma}\label{le:gp1} 
   Let $J_t, t\in [0,1],$ be a  path in $\Jj(\Ss)$ and
  suppose given a smooth
  family $\Si^A_t$ of $J_t$-holomorphic representatives of $A$ all with the same
 decomposition 
$$
 A = \sum_{i=1}^s\ell_i S_i + \sum_{j=1}^p m_j E_j + \sum_{j>p} n_jB_j 
$$ 
as in \eqref{eq:Abir}.  Suppose further
 that the components of $\Si^A_t$ in classes  $E_j$ and $B_j$
are  embedded (though possibly disconnected).
   Then, after perturbing $J_t$ in $\Jj(\Ss)$, we can assume in addition that
for each $t$  that all
 intersections of these components
 with each other as well as with $\Ss$   are $\om_t$-orthogonal.
\end{lemma}

\begin{proof} 
   We first arrange that all intersections are transverse 
which is possible because
such tangencies happen in codimension at least $2$. 
Then these intersections 
occur at a finite number of points $p_{t,i}$
 that vary smoothly with the parameter $t$.   
 Fix $i$, and denote by $C_1^t, C_2^t$  the two branches of $\Ss\cup \Si^A_t$  that meet at $p_{i,t}$, 
 labelled putting the branch that lies  in $\Ss$ first.
 Thus  $C_1^t, C_2^t$ are smoothly varying (local) curves,
 and, using
 a $1$-parameter version of Darboux's theorem, we may choose smoothly varying 
 Darboux charts $p_{t,i}\in U_{t}\overset{\varphi_{t}}{\lra}B^4(\eps)$ such that
 $$
 \varphi_t(U_t\cap C_1^t) = B^4(\eps)\cap \{z_1=0\}, \quad (\varphi_t(0))_*(J_t) = J_0,
 $$
 where $J_0$ is the standard complex structure on $B^4(\eps)\subset \C^2$.
 Moreover, if $C_1^t\subset \Ss$, we may arrange that $ \varphi_t$ takes the fiber at $p_{t,i}$ 
 of the normal bundle to $\Ss$
 to the axis $z_2=0$. 
  By shrinking $\eps>0$ (which we assume small, but fixed) we can also assume that
 the image
 $\varphi_t(C_2^t)\cap B^4(\eps)$ is the graph $z_2=f_t(z_1)$ of some function  
 such that $f_t(0) = 0$ and $df_t(0)$ is complex linear.  If $df_t(0)=0$, the proof is complete.  So we suppose below that
 $df_t(0)\ne 0$.

An obvious $1$-parameter perturbation of $C^t_2$ near $p_t$ provides us with curves $C_t'$ which coincide with $C_2^t$ outside of some small ball and with the graph of $df_t(0)$ near the origin (in the coordinates given by $\phi_t$). Since this perturbation can be made $\Cc^1$-small, $C_t'$ remains symplectic. In other words, we can assume  that there is 
$0<\de<\eps$ such that 
$$
\varphi_t(C_2^t)\cap B^4(\de) = \gr df_t(0) \cap B^4(\de)
=\{(z,a_t\cdot z), \; z\in \C\}\cap B^4(\de),
$$
 where $a_t\cdot $ denotes the  multiplication by the non-vanishing complex number $a_t\approx df_t(0)$. 

Let now  $\rho: [0,\de]\to [0,1]$ be a non-decreasing  cut-off function that equals $0$ near $0$ and $1$  near $\de$, and consider the curves $C_t'':=\{(z,\rho(|z|)a_t z)\}\cap B^4(\delta)$. These curves are  embedded, coincide with $\{z_2=0\}$ near $0$, with $\varphi_t(C_2^t)=\{(z,a_tz)\}$ near $\partial B^4(\de)$, and they are symplectic because $\Jac \bigl(z\mapsto \rho(|z|)a_tz\bigr)=\rho'(|z|)|a_t||z|\ge 0$ (in polar coordinates, $\rho(|z|)a_t z$ is the map $(r,\theta)\mapsto (\rho(r)|a_t|,\theta+\arg a_t)$). 
 We may therefore 
replace $C_2^t\cap \varphi_t^{-1}(B^4(\de))$ by  $\varphi_t^{-1}(C_t'')$.  This is symplectically embedded (and hence $J$-holomorphic for some $\om$-tame $J$), and $\om$-orthogonal to 
$C_1^t$ at $p_{t,i}$.   Finally, if $C^t_1\subset \Ss$ we need to check that the new $C_2^t$ is $J$-holomorphic for some $J\in \Jj(\Ss)$. But this holds because we constructed $C_2^t$ to coincide with the normal fiber to $\Ss$ at $p_{t,i}$.
  \end{proof}

A family of nodal curves $\Si_t$ that satisfies the conclusions of the above lemma for a fixed $\om$ will be called {\bf $\Ss$-adapted.}  
In particular this means that the corresponding homological decomposition of $A$ is fixed, as is the intersection pattern of its components.
The next result gives conditions under which $A$ is represented by a family of embedded curves.

\begin{lemma} \label{le:bir0}
Let $\Ss$ be any singular set and $A$ be $\Ss$-good.
Suppose that for 
 every $J\in \Jj_{semi}(\Ss,A)$ and
every  decomposition
\eqref{eq:Abir} given by a $J$-holomorphic  stable map $\Si^A$ that is a limit of embedded curves
we have
 $d(B )\le d(A)$ with equality only if $B=A$ (so that the decomposition is trivial).
Then: \begin{itemize}\item[(i)]
for each $J\in\Jj_{semi}(\Ss,A)$
there is an
embedded $J$-holomorphic $A$-curve of genus $g(A)$ through
a generic $ \frac 12 d(A)$-tuple of points in $M$ 
\item[(ii)]  any two elements
$J_0, J_1\in \Jj_{reg}(\Ss, A)$ 
can be joined by a path 
$J_t, t\in [0,1],$  in $\Jj_{semi}(\Ss, A)$ for which there is a smooth family 
 of embedded $J_t$-holomorphic $A$-curves.
\end{itemize}
\end{lemma}

\begin{proof} 
Let us first suppose that $d(A)>0$.  
By definition of $\Gr(A)$ (cf. Fact~\ref{f:A}),
there is for each generic $\om$-tame $J$ and
each sufficiently generic set $\bx$ of $\frac 12 d(A)\ge 1$ points in $M$
an embedded $J$-holomorphic curve $u:(\Si,j)\to (M,J)$
that goes through these points, where  $(\Si,j)$ is some smooth Riemann surface of genus $g(A)$.
Hence by Gromov compactness, for every  $\om$-tame $J$ and every set of $\frac 12 d(A)$ points, there is a possibly nodal representative of the class $A$ through these points.
We show below that when $J\in \Jj_{semi}(\Ss,A)$ a generic  set $\bx$  does not lie on a nonsmooth nodal $J$-holomorphic representative for $A$.  Hence, as above, it must lie on an embedded representative.

Consider  a $J$-holomorphic representative  $ \Si^A$  of $A$ with  nontrivial decomposition
 \eqref{eq:Abir}. 
  If we remove the rigid components in the classes $\ell_i S_i$ and $m_j  E_j $ from  $\Si^A$ we are left with
 a  stable map $\Si $ in the class $B = \sum_j n_jB_j$.  
  Since $B_j^2\ge 0$ by Lemma~\ref{le:Abir}, we can 
 resolve all singularities and double points of the components of $B$ as in Case 4 for the proof of Proposition~\ref{prop:gp}, obtaining 
an embedded representative $C^B$ of the class $B$.
   Moreover  $C^B$ is connected
 unless $B^2=0$, and in which  case it has $m$ components in class $B_0$, where $B_0$ is primitive and $B=mB_0$.  In the latter case $d(B) = c_1(B) = md(B_0)$, and in either case $d(B)< d(A)$ by hypothesis.
 But if
 $C^B$ is connected, then we conclude from equation \eqref{eq:ind<d} that  the Fredholm index of a simple connected curve in class $B$ is at most $d(B)$.
 Because $J$ is semiregular,
 each $B$-curve is an element in a moduli space 
 of dimension at most $ d(B)+1$. Hence it cannot go through more than $\frac 12 d(B)<\frac 12 d(A)$ generic points.
   Similarly, if $B=m B_0$, then $d(B) = md(B_0)< d(A)$.
   As above, a $B_0$-curve can go through at most $\frac 12 d(B_0)$ points, so that a $B$-curve  
 goes through at most $\frac m2 d(B_0) = \frac 12 d(B)$ points.  
 This shows that no simple representative of $B$ goes through a generic set $\bx$.  However, 
 as explained in  Remark~\ref{rmk:nodaldim}~(i), the nodal representatives involved by the decomposition \ref{eq:Abir} are 
 even more constrained, because their components satisfy $B_j^2\ge 0$ and $C^{B_j}\cap (M\less \ov\Nn)\neq 0$ 
 by Lemma \ref{le:SiA1}.
  Hence there is no $J$-holomorphic representative of $B$  through $\bx$.
 This completes the proof of (i).

 To prove (ii),  given $J_0,J_1\in \Jj_{reg}(\Ss,A)$,  join them by a regular homotopy $J_t\in \Jj_{semi}(\Ss,A)$ 
 as in Remark~\ref{rmk:reg1}.    Then
the space of $B$-curves that are $J_t$-holomorphic for some $t$ and intersect $M\less \ov\Nn$
  forms a manifold of dimension
$d(B) + 1$.  Hence again we may choose tuple $\bx$ of $\frac 12 d(A)$ points in $M\less \ov\Nn$ that does not lie on any such 
$B$-curve.   Therefore the space of embedded $A$-curves through $\bx$ is a compact $1$-manifold with boundary at $\al = 0,1$.
But because $A$ is $\Ss$-good,   $\Gr(A)\ne 0$.  Hence there is at least one component of this manifold with one boundary at $\al=0$ and the other at $\al = 1$.  Thus for some continuous function $\phi:[0,1]\to[0,1]$ with $\phi(0)=0$ and $\phi(1)=1$
there is a family of 
embedded $J_{\phi(t)}$-holomorphic $A$-curves. 
This proves (ii).

When $d(A) = 0$ the argument is similar.  
In this case the hypothesis means that unless $A=B$ we have $d(B)< 0$.  Since $d(B)$ is even, this means that $d(B)\le -2$.  But then $ \dim  \Coker D_{u,J} \ge 2$ for every $B$-curve $u$.  Hence given  any regular homotopy $J_t\in \Jj(\Ss)$,  the class $B$ has no $J_t$-holomorphic  representatives for any $t$, so that 
all representatives of $A$ must be embedded.  Therefore the previous argument applies.
\end{proof}

The next lemma applies in the situation of Proposition~\ref{prop:1} where
the manifold is rational or ruled and
 we have a smooth family $\om_t$ of $\Ss$-adapted symplectic forms.

 \begin{prop}\label{prop:Z} 
 Let $M$ be a rational or ruled symplectic manifold, and  
let
 $A$  an $\Ss$-good class  with $d(A)>0$.  Suppose further that
$d(A)> g+\frac k4$ if $M$ is the $k$-point blow up of a ruled surface  of genus $g$. 
Let $J_t\in \Jj(\Ss, \om_t, A), t\in [0,1]$ be a smooth path with endpoints in $\Jj_{reg}(\Ss, A)$
Then, possibly after reparametrization with respect to $t$, the path
 $(J_t)_{t\in [0,1]}$ can be perturbed to a smooth $\Ss$-adapted 
 path $
 \bigl(J_t'\in \Jj_{semi}(\Ss, \om_t, A)\bigr)_{t\in [0,1]}$
  such that there is   a smooth family $\Si_t^A, t\in[0,1],$  of $J_t'$-holomorphic and 
$\Ss$-adapted   
nodal curves in class $A$.   Moreover the corresponding decomposition
$$
A=\sum_{i=1}^s \ell_{i} S_i +\sum_j m_{j} E_{j}+ B, \hspace{.5cm} E_{j}^2=-1,\
$$
of \eqref{eq:Abir} has  $\Gr(B)\ne  0$.
\end{prop}

\begin{proof} {\bf Step 1:} {\it Preliminaries.}\
Because we are in dimension $4$, $(M,\om_t)$ is semi-positive in the sense of \cite{MS}.  Hence by the results of
\cite[Chapter~6]{MS}  
we may join $J_0, J_1$ by a 
regular homotopy $J_t\in \Jj(\Ss,\om_t,A)$.  As in Remark~\ref{rmk:reg1}, this means that
$\p_t J_t$ covers the cokernel of $D_{u,J_t}$ for every relevant map $u$, so that  the moduli spaces 
$\bigcup_{t\in [0,1]} \Mm(M\less \ov\Nn,B,J_t)$ are smooth manifolds with boundary of the \lq\lq correct" dimension.
In particular each $J_t\in \Jj_{semi}(\Ss,\om_t, A)$.

Given such $J_t$, consider the following compact space of stable maps:
$$
X: = \bigcup_{t\in [0,1]} \oMm(A,J_t).
$$
This space is stratified according to the topological type $\Tt$ of the domains of the stable
maps, where $\Tt$ keeps track both of the structure of the domain and the homology classes of the corresponding curves.
These strata $X_\Tt$ are ordered by the relation that $\Tt'\le \Tt$ if a stable curve with domain of type $\Tt$ can degenerate into one of type $\Tt'$.     Since $J_t$ ranges in a compact set there are a finite number of such decompositions
$A = \sum_i\ell_iS_i + \sum_j m_jE_j + B$ as in \eqref{eq:Abir}. 
Let $d_{\max}$ be the maximum of the numbers $d(B)$, where $B$ occurs in such a decomposition for some $t\in [0,1]$.
We claim that $d_{\max}\ge d(A)$.  For otherwise
Lemma~\ref{le:bir0} implies that for each $t$
there are embedded $J_t$-holomorphic  $A$-curves.  Since this is one of the decompositions considered in the definition of $d_{\max}$, we must have   $d_{\max}\ge d(A)$. 

Next, consider a decomposition
\begin{equation}\label{eq:fix}
A = \sum_{i=1}^s\ell_i S_i + \sum_{j=1}^e m_jE_j + B
\end{equation}
of the given type with $d(B) = d_{\max}$ and with  maximal
multiplicities $(\ell_i)$, in the sense  that there is no other representative of
$A$ with decomposition  $\sum_{i=1}^s\ell_i' S_i + \sum_{j=1}^{e'} m_j' E_j' + B'$ where
$d(B')=d_{\max}$, $\ell_i'\ge \ell_i$ for all $i$ and $\ell_i'>\ell_i$ for some $i$.
\MS

\NI
{\bf Step 2:} {\it In this situation,we have  $\Gr(B)\ne 0$.}
  If $M$ is rational, this follows from Lemma~\ref{le:SW}~(i) since $B^2\ge 0$, $\om (B)>0$ by construction, and  $d(B)\ge d(A)\ge 0$.   So suppose that
$M$ is the blowup of a ruled surface.
If $B^2=0$ then 
we may write $B=m B_0$ where $m\ge 1$ and $B_0$ is represented by an embedded curve.  This must be a sphere or torus, since in all other cases the Fredholm index of the class is $\le -2$, so that
by definition of $\Jj_{semi}$ they are not represented.
In the case  of a sphere we have $\Gr(B)=1$, since for generic $J$ there is a unique $B$-curve through each set of $m$ generic points. On the other hand, in the case of a torus $d(B_0) = d(B)= 0$.  
Since $d(B)=d_{\max} \ge d(A)> g+\frac k4> >0$ by hypothesis this case does not occur.
Therefore, it remains to consider the case when
 $B^2>0$.  Since $\om(B)>0$ by construction,  this means that $B\in \Ppp^+$. Therefore $\Gr(B)\ne 0$
 by  Lemma~\ref{le:SW}~(ii) applied to the class $B$.
\MS

\NI
{\bf Step 3:} {\it   Completion of the proof.}
\MS

Since
the classes $E_j, B$ in \eqref{eq:fix} have nontrivial Gromov invariant, they are always represented in some form for each $J_t$. By Proposition~\ref{prop:Ee} the classes $E_j$ are
 in fact always represented by embedded curves $C^{E_j}_t$ when 
 $J\in \Jj_{semi}(\Ss,\om_t,A)$ since  $J\in \Jj_{semi}(\Ss,\om_t,A) \subset \Jj_{semi}(\Ss,E_j)$.
We next  check that we can choose the 
regular homotopy $J_t'\in \Jj_{semi}(\Ss,\om_t,A)$ so that the class $B$ does not decompose.
As in the proof of Lemma~\ref{le:bir0}, this will follow if we can show
  that for each decomposition $B = \sum_j B_j'$ of the
 $B$-curve, the sum of the Fredholm indices of its nonrigid components is strictly less than
 the Fredholm index $d(B)$ of the class $B$.
  If the components of the $B$ curve are all transverse to $\Ss$, then this calculation is standard; 
  cf.~Remark~\ref{rmk:nodaldim}~(i).
 On the other hand, if for some $J_t'$ the decomposition  is a stable map $(\Si_B)'$ that involves some components of $\Ss$ with others in class $B'$,
then the maximality of the pair $d(B)=d_{\max}$ and $(\ell_i)$
 implies that $d(B')< d(B)$, and since   $d(B')$ is always even, we actually have
 $d(B')\le  d(B)-2$. Therefore in a regular path $J_t'$, the dimension of the moduli
 space of these stable  maps is at most $d(B)-1$, and hence those curves cannot go through
 $k: = \frac 12 d(B)$ generic points.\footnote
 {
 Strictly speaking, we can only control the dimension of the family of curves that go through some point in $M\less \ov \Nn$. However, since $d(A)>0$ we only need consider curves that go through at least one fixed point that we can choose far from $\Ss$.}
Thus,  the space of embedded $B$-curves that are $J_t'$-holomorphic for some
 $t$ and go through $k$ generic points 
 is a compact $1$-manifold with boundary.  Moreover, because $\Gr(B) \ne 0$ there is at least one connected component of this  $1$-manifold with one end at $t=0$ and the other at $t=1$.  Taking such a component, and reparametrizing with respect to $t$ as necessary, we therefore have a family $C^B_t$ of embedded $J_t'$-holomorphic curves in class $B$.   
\end{proof}

\begin{rmk}\label{rmk:t} \rm  The above proposition constructs $1$-parameter families of nodal curves whose components are covers of embedded curves.
Lemma~\ref{le:gp1} shows that if we start with a family of nodal curves  whose components are  embedded (or immersed)  
we can perturb them so that they intersect $\om_t$-orthogonally.   The patching arguments in Proposition~\ref{prop:gp} that resolve double points and amalgamate transversally intersecting  components in classes $B, B'$ with $B^2, (B')^2\ge -1$ also work for  $1$-parameter families.    
Therefore, we can apply
Proposition~\ref{prop:gp} to these $1$-parameter families of nodal curves.  (The only part of this proposition that might fail in a $1$-parameter family is the initial resolution of singularities.)
\end{rmk}

\begin{cor}\label{cor:Z1} Proposition~\ref{prop:1} holds when $\Ss_{sing}=\emptyset$.
\end{cor}
\begin{proof} This holds by applying the $1$-parameter version of 
Proposition~\ref{prop:gp}  as in Remark~\ref{rmk:t}
to obtain the required family of embedded curves.  
\end{proof}

\subsection{Numerical arguments}\label{ss:num}

This section proves Theorem~\ref{thm:1} under hypotheses 
(iv) and (v) by showing in both cases that the hypotheses in Lemma~\ref{le:bir0} are satisfied.
First we discuss the genus zero situation, using an argument adapted from Li--Zhang~\cite[Lemma~4.9]{LZ}.\footnote{Instead of requiring $J$ to be in some way generic, they use
the hypothesis that $A$ is $J$-NEF, which also implies that $A\cdot B_j\ge0$ for all $j$.}

\begin{lemma}\label{le:LZ}
Let $\Ss$ be any singular set.  Then the hypothesis of  Lemma~\ref{le:bir0} holds for  every 
$\Ss$-good $A$ such that 
$$
g(A): =1+\tfrac 12 (A^2-c_1(A))=0,\qquad d(A): = A^2+ c_1(A)> 0,
$$
and every $J\in \Jj_{semi}(\Ss,A)$.
\end{lemma}
\begin{proof}  
Note first that we must be in the situation  $b_2^+=1$, since
 by Fact \ref{fact:grdb},  $\Gr(A)\ne 0$ can only be consistent with 
$d(A)>0$ in this case.  
 Consider a nontrivial decomposition  $A=\sum \ell_iS_i+\sum m_jE_j+ B$ as in Lemma~\ref{le:Abir},
  given by a 
 $J$-holomorphic  stable map $\Si^A$ with $J\in \Jj_{semi}(\Ss,A)$ that, by construction, is the limit of embedded 
 $A$-curves.  We must check that $d(B)<d(A)$.

 Let us suppose first  that 
 $B$ has a connected, smooth and somewhere injective $J$-holomorphic representative $u:(\Si,j)\to (M,J)$. 
Then the adjunction formula~\eqref{adj} implies that  $g(B) \ge g_\Si\ge 0$,\footnote{Although we know each $g(B_j)=0$, it is a priori possible that $g(B)>0$. Li--Zhang's argument shows that in fact this does not happen. However, we do not need to use this.}
so that
$$
\tfrac 12 d(B) = 1+B^2 - g(B)\le 1+B^2.
$$ 
Thus $\frac 12 d(B ) \le 1+B^2$ while our
 hypotheses imply that
 $\frac 12 d(A) =   1+A^2$.
  Thus   it suffices to show that
 $B^2 <A^2$.  But
\begin{eqnarray*}
 A^2 -B ^2 = (A+B )\cdot(A-B )& =& A\cdot\bigl(\sum \ell_i S_i + \sum n_j E_j \bigr) + B \cdot (A-B )\\
& \ge& B \cdot (A-B )> 0,
\end{eqnarray*}
where the first inequality holds because $A$ is $\Ss$-good, 
and the second (strict) inequality holds because, as we noted above, $\Si^A$ is connected and $A\ne B $.

By  Fact~\ref{f:A},
 this completes the proof unless $B =nB _0$ where $B_0$ is a primitive class with $B _0^2 = 0$.  Since $g(B_0) = 0$ by construction, Lemma~\ref{le:Jreg}~(ii) implies that
each $B_0$-curve is an embedded sphere.  Hence $\frac 12 d(nB_0 )=n$. 
Thus we need to show that $\frac 12 d(A) = 1+ A^2> n$.  But this holds because, by the above calculation
$$
A^2 = A^2-B^2\ge B\cdot (A-B) \ge n,
$$
where the last inequality holds because  $B $ has $n$ disjoint components.
\end{proof}

We next extend an argument from Biran~\cite{Bir}. Recall that $\Ss_{sing}$ consists of all the negative components of $\Ss$ that are not exceptional spheres.

\begin{lemma}\label{le:Bir}  Let $\Ss$ be any singular set 
such that $c_1(S_i)=0$ for all $i$ with $C^{S_i}\in \Ss_{irreg}$. 
Then the hypothesis of  Lemma~\ref{le:bir0} holds for every 
$\Ss$-good
class  
$A\notin \Ee$ 
such that $A\ne \sum_{i\in \Ss_{sing}}\ell_i S_i$ 
and every $J\in \Jj_{semi}(\Ss,A)$.
\end{lemma}
\begin{proof}    Starting with a nodal curve $\Si^A$ with decomposition with $d(B)\ge 0$ as in 
Lemma~\ref{le:Abir}, add to $B$ 
all regular components $C^{S_j}$ and all exceptional spheres that intersect $B$ as in Lemma~\ref{le:gp2}.  As we remarked  in the proof of Lemma~\ref{le:gp2}~(iv), $d(B)$ does not decrease when we do this.  
By hypothesis  all irregular components of $\Ss$ are negative 
 because they  have $d(S)=S^2+c_1(S)<0$.   
Therefore it suffices to show that in any decomposition
 $$
A = \sum_{i\in \Ss_{neg}}\ell_i S_i + \sum_k m_k E_k + B
$$
we have $d(B)<d(A)$.  Rewrite this decomposition as
$$
A = \sum_{i\in \Ss_{sing}}\ell_i S_i + \sum_j n_j E_j' + B,
 $$
where we have grouped the sums $ \sum_{i\in \Ss_{neg}\less \Ss_{irreg}}\ell_i S_i$ and
$\sum_k m_k E_k$ into a single sum over classes $E_j'\in \Ee$.
We write
 $Z: = \sum_{i\in \Ss_{sing}} \ell_i S_i$, and note that by hypothesis  $c_1(Z) = 0$.
  
First suppose that $B =0$ so that $A = Z+\sum_j n_j E_j' $.
We must show that $d(A)>0 = d(B )$.
By assumption
 $A\ne Z$.
Further,
$$
d(A) = A^2 + c_1\bigl(Z + \sum_j n_j  E_j'\bigr) = A^2 + \sum n_j >0
$$
unless  $A\in \Ee$ and $\sum n_j  = 1$.  But we excluded the case $A\in \Ee$.
   Hence when $B =0$ we have $d(A)>0$ as required.

Now suppose that $B \ne 0$. 
By Lemma~\ref{le:gp2}, we may assume that $B\cdot E_j'=0=E_j'\cdot E_k'$ for all $j\ne k$.
Further $(A-B)\cdot B > 0$ since the classes $A-B$ and $B$ are both
represented by $J$-nodal curves with no common component, and their union in class $A$ is connected. 
Hence
\begin{eqnarray*}
d(A)-d(B) &= & \bigl(Z+\sum n_jE_j'+B\bigr)\cdot A-{B}^2 +c_1(Z+\sum n_jE_j')\hspace{.05in}\\
& = & \bigl(Z+\sum  n_jE_j'\bigr)\cdot A+B\cdot(A-B)+\sum 
n_jc_1(E_j')+c_1(Z)\hspace{.05in}\\
 & >&Z\cdot A +\sum
 n_j \geq 0,
\end{eqnarray*}
where the strict inequality uses the fact that  
$B\cdot(A-B)>0$. This completes the proof. 
\end{proof}

\begin{cor}\label{cor:45} Parts (iv) and (v) of Theorem~\ref{thm:1} hold.
\end{cor}  
\begin{proof}
If $A\in \Ee$ the result follows from Proposition~\ref{prop:Ee} . Therefore we will assume $A \notin\Ee$.
To prove Theorem \ref{thm:1}~(iv)  notice that 
$d(A)\geq 0$
because $A$ is $\Ss$-good. Moreover
 when $g(A)=0$, the equality $d(A)=0$ implies that $A^2=-1$, so that $A\in \Ee$, contrary to hypothesis.  
 Thus $d(A)>0$.
But then
Lemma~\ref{le:LZ} combined with Lemma \ref{le:bir0} shows   that $A$ has an embedded $J$-representative for $J\in \Jj_{semi}(\Ss,A)$.  
Since $\Jj_{semi}(\Ss,A)$ is residual
 by Lemma \ref{le:reg} (ii), this proves part (iii) of  Theorem~\ref{thm:1}.
Part (v), again with  $\Jj_{emb}(\Ss,A)=\Jj_{semi}(\Ss,A)$, follows similarly using  Lemmas~\ref{le:Bir} and \ref{le:bir0}.
\end{proof}

\begin{rmk}\label{rmk:bir1}\rm 
Biran actually assumed the weaker condition $A\cdot S_i + c_1(S_i)\ge 0$ for all $i$, but worked with disjoint curves $C^{S_i}$. 
\end{rmk}

 \section{Constructions}\label{s:geom}

In \S\ref{ss:geom} we explain some geometric constructions for embedded curves, and then
prove part~(ii) of Theorem~\ref{thm:1} and the second case of Proposition~\ref{prop:1}.   
The asymptotic result Theorem~\ref{thm:asympt} is proved in  \S\ref{sec:asympt}.

 \subsection{Building embedded curves by hand}\label{ss:geom}

The naive strategy for answering Question~\ref{q:1} is to take 
the nodal curve $\Si_A$ and try to piece its components together.  
A basic tool on which this strategy builds on is the following easy 
patching 
lemma. 
 
\begin{lemma}\label{le:patching}  Suppose that the integers  
$\ell,m > 0$ 
have no common divisor $>1$. 
Given two  
nonvanishing and holomorphic functions $h_1,h_2$ in a \nbd\ of $0\in \C$
 and $\eps>0$ small enough,
 there is an embedded symplectic submanifold 
 $$
 C_f: = \{f(z,w) = 0\}\subset  \C^2\less \{zw=0\}
 $$
  which coincides with $\{w^\ell=
\eps
 h_1(z)z^{-m}\}$ on $|z|<\eps_1$ and with $\{z^m=
 \eps
 h_2(w)w^{-\ell}\}$ on $|z|>\eps_2$, and is disjoint from the axes.
\end{lemma}

Note that when $\ell=m=1$ we are patching the graph of a meromorphic section $w=a z^{-1}$ over the $z$-axis to the graph of a meromorphic section $z=b w^{-1}$ over the $w$-axis via the cylinder  $C_f$.
Similarly, one can patch two transversally intersecting curves, and also a simple pole (the graph of $w=a z^{-1}$) to the 
transverse axis $z=0$. In the latter case, for example, $C_f$ would coincide with the graph of  $w=a z^{-1}$ for $|z|$ large and with 
the axis $z=0$ for $|w|$ large.
We will not prove this lemma here (or state it very precisely) since we do not 
use it in any serious way in this paper.
However, we describe some applications in Example~\ref{ex:pole2} below. Note that Li--Usher~\cite{LU} also 
use this idea of patching curves via
meromorphic sections.

The way this lemma would ideally apply is the following.  To fix ideas,  consider the case where the $S_i$ are spheres of self-intersection $-k_i\le -2$. For the decomposition $A=\sum \ell_iS_i+\sum m_iE_i+B$ associated to a nodal map $\Si^A$, the numerical condition $A\cdot S_i\ge 0$ implies that 
\begin{equation}\label{eq:ellk}
\ell_ik_i\le \sum_{j\neq i} \ell_jS_i\cdot S_j+ \sum m_j E_j\cdot S_i+B\cdot S_i. \tag{$*$}
\end{equation} 
We consider a holomorphic  cover $\Sigma_i\overset{f_i}{\lra} S_i$ of degree $\ell_i$, totally ramified at the intersections between $C^{S_i}$ and each $C^{S_j}$ and $C^{E_j}$. We pull back the normal bundle $\Ll_i$ to $C^{S_i}$ by $f_i$, and  consider a smooth section $\sigma_i$ of $f_i^*\Ll_i$ 
that is holomorphic near its zeros and poles,  has
 poles of order  $\ell_j, m_j$
  at each (unique) preimage of the intersections of $C^{S_i}$ with $C^{S_j}$ and $C^{E_j}$, respectively, as well as one additional simple pole at some preimage of each intersection of $C^{S_i}$ with $C^B$, and no other poles.  Since the pullback bundle $f_i^*\Ll_i$ has degree $-\ell_ik_i$, the condition 
\eqref{eq:ellk}
 precisely means that the existence of  such smooth sections is not homologically obstructed. 
 We do the same for $C^{E_i}$ and for $C^B$ (for the latter we do not need to consider a covering).
 Now the push-forward of these sections to $\Ll_i$ provide multi-sections with singularities modelled on 
  $w^{\ell_i}=z^{-\ell_j}$ (or $w^{\ell_i}=z^{-m_j}$ or $w^{\ell_i}=z^{-1}$)  
 near each intersection. 
For example, at an intersection $q\in  C^{S_i}\cap C^{S_j}$ 
let us use the coordinate $z$ along $C^{S_i}$ and $w$ along $C^{S_j}$.  Then the two branched covering maps 
are
$$
\bigl(z', w\bigr)\mapsto \bigl((z')^{\ell_i}=z, w\bigr) , \qquad \bigl(z, w')\mapsto \bigl(z, (w')^{\ell_j}=w\bigr).
$$
Hence  the sections $w= a \,
(z')^{-\ell_j},\ z=b \,
(w')^{-\ell_i}$ push forward to the curves
$$
w^{\ell_i} = a^{\ell_i}\ z^{-\ell_j}, \qquad  z^{\ell_j} = b^{\ell_i}\ w^{-\ell_i}.
$$
Thus Lemma \ref{le:patching} implies that for sufficiently small $\eps$ 
the sections
$\eps f_{i*}\sigma_i$ and $\eps f_{j*}\sigma_j$ 
can be patched together to give a curve that does not meet  $C^{S_i}\cup C^{S_j}$ near the intersection point $q$.   
More generally, all these  (rescaled) 
multi-sections can be patched together in the \nbd\ of the intersections to form a symplectic curve in class $A$
 that is transverse to $\Ss$.  
 
 Now this 
curve may have self-intersections coming from the folding of the section $\sigma_i$ when we push it forward to $\Ll_i$. When $\sigma_i$ is holomorphic, these self-intersections are positive, so they can be resolved and the procedure gives an embedded symplectic curve in class $A$ that intersects the $\Ss$ transversally and positively. However,
the criteria for the existence of such a holomorphic section is not of topological nature but of analytical one (it is given by the Riemann-Roch Theorem).  Hence there is no guarantee 
that one can find suitable sections $\si_i$.  
The next example illustrates these difficulties, which in this case
arise from a multiply covered exceptional curve $C^{E}$.  It also suggests some ways  around them.

\begin{example}\label{ex:pole2} \rm    Suppose that $\Ss$ consists of a single sphere $C^S$ in class $S$ with
$S\cdot S = -k$, that $E$ is the class of an exceptional divisor $C^E$ with $E\cdot S = m$ and that $B$ satisfies
$B\cdot S = 1, B\cdot E = 0$.  Then $A: = S+ m E + B$ has
\begin{gather*}
A\cdot E = 0,\quad A\cdot S = m^2-k + 1,\\
 d(A) = d(S+mE) + d(B) + 2 S\cdot B = 4-2k + m^2+m + d(B)\ge 0.
\end{gather*}
Because $d(B)$ can be arbitrarily large, the condition $d(A)\ge 0$ gives no information.  Therefore, the only numerical information we have on $k$ is that $k\le m^2 + 1$.  Note also that if $k\le m^2$, then $A': = S+mE$ satisfies $A'\cdot S = 0$, and
we can try to form an embedded curve in class $A' = S+mE$ and then join it to the $B$ curve to get the final embedded $A$ curve.  The virtue of this approach is that it gives us better understanding of the genus since $g(A')$ is a function of $m$ only.
In fact,  because $A'\cdot B = 1$, we have $$
g(A'+B)  = g(A') + g(B),\mbox{ and } g(A') = 1+\tfrac 12((A')^2-c_1(A')) = \tfrac 12 m(m-1).
$$
Therefore if $m=4$ and $k\le 16$, we should be able to construct an embedded curve in class $A' = S+4E$ 
of genus $6$ and hence a curve in class $A$ of genus $6 + g(B)$.
We show below that the embedded $A'$-curve exists when 
$k\le 13$, but may not exist when $14\le k\le16$.
\MS

\NI {\bf The case $k\le 4$:}   In this case it is very easy  to construct such a curve.  We may assume that $C^E$ intersects $C^S$ transversally at $4$ distinct points $p_1,\dots,p_4$, and then choose a small meromorphic section $\si_S$ of the normal bundle to $C^S$  with simple poles at the four points $p_1,\dots,p_4$ and $4-k$ zeros.  Then take $4$ different small nonvanishing
meromorphic  sections $\rho_1,\dots,\rho_4$  of the normal bundle to $C^E$, where $\rho_i$ has a simple pole at $p_i$.  Note that these sections are inverse to holomorphic sections of the bundle over $S^2$ with Chern class $1$ and so each pair intersects once transversally.   Next patch $\rho_i$ to $\si_S$ at $p_i$.  (This is possible because the graphs of $\si_S$ and $\rho_i$
satisfy an equation of the form $zw=const.$  near $p_i$ and so we can cut out small discs from each of these graphs and replace it by a cylinder.   One needs to check that this cylinder can be chosen to be disjoint from
the other sections $\rho_j$; but this holds because $\rho_i$  is relatively much larger than  the $\rho_j, j\ne i$ near $p_i$ since it has a pole there.)  This process gives an immersed curve of genus $0$ with $6$ positive self-intersections, 
one for each (unordered) pair $i,j, i\ne j$.\footnote{
These intersection points occur at the places where the graphs of $\rho_i,\rho_j$ intersect, far away from the poles.
Note that although the graph of each $\rho_i$ meets $C^S$ at $3$ points, one near each $p_j, j\ne i$,  these intersections disappear after gluing since  the part of $C^S$  near $p_j$ is cut out during the gluing
with $\rho_j$.
}
Therefore we obtain the desired embedded curve of genus $6$ by resolving these intersections.
\MS

\NI {\bf The case $4<k\le 10$:}   
We can refine the above argument by choosing the sections $\rho_1,\dots,\rho_4$ to have different orders of magnitude,  
with $\rho_1\gg \rho_2 \gg\rho_3\gg\rho_4$. Thus $\rho_1$ has a simple pole at $p_1$ and 
and its graph intersects $C^S$
at points $p_{1i}, i=2,3,4$ moderately near $p_i$.
We match 
these
zeroes and poles
 with $4$ poles of $\si_S$.  
If $\rho_2$ is much smaller than $\rho_1$, then we can construct $\si_S$ to have another pole at $p_2$ that matches with $\rho_2$, together with two more poles at $p_{2i}, i=3,4$ that are much closer to $p_i$.  (Note that the point of intersection of  the graph of $\rho_2$ with $C^S$ that is near $p_1$ is cut out of $C^S$ by the first patching process, and so we cannot put another pole there.) Similarly, we can choose $\rho_3$ so that it patches with $2$ poles of $\si_S$ and then can take $\rho_4=0$ to patch with one further pole.   This procedure accommodates up to $10$ poles.  
Moreover, it is not hard to check that the corresponding embedded curve has genus $6$.  
For example if $k$ is $10$ we have patched five spheres together at 
$4+3+2+1=10$
 points and so get a possibly immersed curve of genus $6$.  As before,  the branch points would  come from intersections of $\rho_i$ with $\rho_j$ for $i\ne j$.  But these are all cut our during the patching process:  for example, because $\rho_2<<\rho_1$ the intersection point of these sections 
lies near the pole on $\rho_2$ and so is cut out when this pole is patched to the pole of $\si_S$ at $p_2$. 
\MS

\NI {\bf The case $10< k\le 16$:}  It is possible to refine this argument by using using branched coverings
as suggested at the beginning of this section.
Note that near a point where $\si_S$ has a pole of order $n$   its graph satisfies
an equation of the form $wz^n = const$, where $z$ is the coordinate along $C^S$ and $w$ is the normal coordinate.  It is not hard to check that this pole may be patched to the pushforward of a section $\rho$ with a simple pole $\rho(w')=\eps/w'$  by a branched covering map
  $w'\mapsto (w')^n: = w$: indeed the graph of $\rho$ satisfies $zw' = \eps$, which gives 
  $z^n(w')^n = \eps^n$, so that its pushforward
satisfies $z^n w = \eps^n$.   Since $A$ contains  $E$ with multiplicity 
$4$,
we can in principle take any $n\le m=4$
and hence accommodate up to $16$ poles of $\si_S$.  
 We now investigate this construction in more detail.
\MS

\NI {\bf The case $k>10$:}
Our initial strategy for constructing a curve in class $A'=S+4E$ when $k>4$  is the following:
\begin{itemize}
\item[(a)] take a meromorphic section $\si_S$ of the normal bundle to $C^S$ with poles of order $4$ at each
point $p_1,\dots,p_4$ and $16-k$ zeros;
\item[(b)] take a branched cover $f:\Si\to S^2$ of $C^E$ of order $4$ that is totally ramified at each of the points $q_i: = f^{-1}(p_i), i=1,\dots,4$ (and hence has local model $w'\mapsto (w')^4$);
\item[(c)] choose a meromorphic section $\rho_\Si$ over $\Si$ of the pullback by $f$ of the normal bundle to $C^E$ with simple poles
at the branch points $q_1,\dots,q_4$;
\item[(d)] patch the multisection $f_*\big({\rm graph\,}\rho_\Si\bigr)$ to the graph of $\si_S$ obtaining an immersed curve with only positive intersections with itself and with $\Ss$.
\end{itemize}
Step (d) gives an immersed curve which is made by patching an immersed curve of genus $g(\Si)$ with $4$ punctures to a sphere with $4$ punctures.  Hence it has genus 
$g(\Si)+3+a$, where $a$ is the number of self-intersection points of 
$f_*\big({\rm graph\,}\rho_\Si\bigr)$.  By the Riemann-Hurwitz formula, the Euler characteristic $\chi(\Si)$
equals $4\chi(S^2) - 12 = -4$, where $12 = 4\times 3$ is the number of \lq\lq missing vertices".  Therefore $g(\Si) = 3$.  Hence if this process worked we would have $a=0$.
Thus the curve in (d) would actually be embedded.
It is not hard to see that all the above steps can be achieved except (possibly) for (c).  The problem here is
that because
 $\Si$ is not a sphere there is no guarantee that we can find a meromorphic section with  poles at the given points.  Here are some ways to try to get around  this problem.

  \begin{itemlist}
 \item Relax the condition on the section in (c), simply choosing any section with these poles.  But then
 there is no guarantee that the pushforward multisection $f_*\big({\rm graph\,}\rho_\Si\bigr)$
 has only positive self-intersections.   In fact, in cases where we have tried this, we have managed only
 to construct  sections with simple poles at the $q_i$  that
  push forward to multisections with both positive and negative self-intersections; and it is not clear that these can be made to cancel.
\item  Change the cover in (b) so that $g(\Si)$ is smaller, since then we can prescribe the positions of
$4-g(\Si)$ poles of $\rho_\Si$.  Suppose for example that $f$ has
three branch points $q_1,\dots,q_3$ of orders $b_i = 4, i=1,2$ and $b_3=3$.
Then the Riemann-Hurwitz formula gives
$$
2-2 g(\Si) = \chi(\Si) = 4  \chi(S^2) - \sum_{i=1}^3\ (b_i-1) = 0,
$$
so that $g(\Si) = 1$.  Moreover there is a cover with this branching because there are three elements
$\ga_1,\dots,\ga_3$ in the symmetric group $S_4$ on $4$ letters such that
\begin{itemize} \item[-] $
\ga_i$ has order $b_i$, for all $i$;
\item[-] $\ga_1\ga_2\ga_3 = \id$.
\end{itemize}
(Take $\ga_1,\ga_2$ to be  cycles of order $4$ whose product fixes just one point and hence is a cycle of order $3$.)
Choose a meromorphic section $\rho_\Si$ with simple poles at the branch points $q_1, q_2, q_3$ and at one other arbitrary point $v_4$.
Then alter $f$ by postcomposing with a diffeomorphism $\phi:S^2\to C^E$ so that $\phi\circ f:\Si\to C^E$ maps the four points $q_1,\dots,v_4$ where $\rho_\Si$ has poles to the intersection points $\{p_1,\dots,p_4\}=C^S\cap C^E$.
Then one can check that the pushforward of $\rho_\Si$ by $\phi\circ f$ can be patched  to a section $\si_S$ with poles of order $4$ at $p_1,p_2$, of order $3$ at $p_3$ and of order $1$ at $p_4$, a total of $12$ poles.
Since the other branch points of $f$ just push forward to smooth points, the result is an  immersed curve with genus $g(\Si) + 3 = 6$ which in fact must be embedded.  

It is not hard to check that this is best one can easily do 
with this approach: adding more branching increases $g(\Si)$ and hence 
decreases the number of points where one can allow $\si_S$ to have higher order poles.
However, in this case it is possible to accommodate one more pole, because there happens to be a special $3$-fold cover $f:T^2\to S^2$
totally ramified over three points, say $p_2,p_3,p_4$:
 see Remark~\ref{rmk:tori}.  We may therefore take a largish section $\rho_1$ of the normal bundle to $E$ with a pole at $p_1$ 
  whose graph intersects $C^S$ at three points close 
  to $p_2,p_3,p_4$, and a very small pushforward multisection $f_*(\si_T)$ that patches to poles of order $3$ at $p_2,p_3,p_4$.  This patches $13$ poles.  However, it is not clear how to deal with the cases $14\le k\le 16$. 
  \end{itemlist}
\end{example}

\begin{rmk}\label{rmk:tori}\rm  We now briefly describe the special  $3$-fold branched cover $f:T^2\to S^2$.  It has three totally ramified branch points $q_1,q_2,q_3$  such that the pullback bundle  has a meromorphic section with its three poles precisely at $q_1,q_2,q_3$.  
Consider the torus $\TT_0$ given by the Fermat curve $x^3 + y^3 + z^3 = 0$ in $\CP^2$, with deformations $\TT_\eps:= x^3 + y^3 + z^3 = \eps xyz$.  There is a natural degree $9$ cover
	 $$
	F: (\C P^2, \TT_0) \to (\C P^2, \C P^1), \quad [x:y:z]\mapsto [x^3:y^3:z^3],
	 $$
	 which quotients out by the action of the group $\Z_3\times \Z_3$  on $\TT_0$ by 
	$$
	[x:y:z]\mapsto [\tau^i x: \tau^j y:z],\quad i,j\in \Z_3.
	$$ 
	 The action of the subgroup $G_{free}: = \{(j,-j), j\in \Z_3\}$ has no fixed points, in fact acting on all the tori  $\TT_\eps$ by a translation of order $3$. 
	 Therefore the map $F: \TT_0\to \C P^1$ descends to
	 $$
	f:\Si: = \TT_0/ G_{free} \to \C P^1.
	 $$
	 On the other hand the group $G_{fix}: = (j,0), j\in \Z_3$ fixes the three points
	 $$
	  [0: 1:-1], [0:\tau: - \tau^2], [0:\tau^2: -\tau],
	  $$
	  acting on the tangent space of each by a rotation through $2\pi/3$.  These points form one orbit under $G_{free}$.  Hence this gives one totally ramified point of $f$ in 
	  $\Si$.    Similarly, $G_{free}$ permutes the three points
	  $ [1:-1:0], [\tau: - \tau^2: 0], [\tau^2: -\tau: 0]$ and the corresponding set of points with $0$ in the second place. Again, each of these gives rise to one totally ramified point in
	  the quotient cover $f$.  Note that $F$ has $9$ branch points, each of order $3$, lying in three distinct fibers of the quotient map $\TT_0\to \Si: = \TT_0/G_{free}$. 
	  
	  One can see the section as follows.  
	The normal bundle $\Ll_N$ to $\TT_0$ in $\C P^2$ is the pullback by $F$ of the normal bundle of the line  $x+y+z=0$.  The $9$ branch points of $F$ lie on all the curves $\TT_\eps$. Define a section $Y_\eps$ of $\Ll_N$  by first embedding a neighbourhood of the zero section in the normal bundle of $\TT_0$ into $\C P^2$ using the exponential map with respect to the standard K\"ahler metric, and then defining $Y_\eps$ so that
	$\exp_z(Y_\eps(z)) \in \TT_\eps$ for all $z\in \TT_0$.  Then its derivative $\p_\eps Y_\eps|_{\eps=0}$ is a holomorphic section of the normal bundle. Thus this is a holomorphic section of $\Ll_N$ with precisely $9$ simple zeros at the branch points of $F$.
To get the bundle and section we are looking for, it remains to quotient out by $G_{free}$, which acts on the curves $\TT_\eps$ and also by isometries on $\C P^2$.  
\end{rmk}

\begin{rmk}\label{rmk:alt} \rm   It is not clear how special the section in Remark~\ref{rmk:tori} really is.  
Are there cases in which there are no meromorphic sections with the required zeros, 
but there are  symplectic  sections with these zeros whose pushforward  has positive self-intersection? 
If so, the local structure of symplectic nodal curves would be significantly different from that of 
holomorphic ones.
\end{rmk}

\begin{prop}\label{prop:ii} Theorem~\ref{thm:1}~(ii) holds.
\end{prop}
\begin{proof} 
By assumption $\Ss$ has one class $S$ with $S^2<-1$, some classes labelled by $i\in \Ii_\Ee$ with $C^{S_i}$ an exceptional sphere, and classes $S_i$ with $(S_i)^2\ge 0$.
By Lemma~\ref{le:gp2} we may assume that $A$ has a connected nodal representative 
 $\Si^A$ with
 decomposition
\begin{equation}\label{eq:dec1}
A = \ell S + \sum_{i\in I_\Ee} \ell_i S_i + \sum_j m_j E_j + B,
\end{equation}
as in equation \eqref{eq:Abir1}, where
\begin{itemize} \item[(I)] $A\cdot S \ge 0$,  $A\cdot E_j\ge 0$, $A\cdot S_i\ge 0$  for all $i,j$
with nonzero coefficients;
\item[(II)]  $S_i\cdot E_j = S_i\cdot B = E_j\cdot B = 0$ for all $i,j$ with nonzero coefficients;
\item[(III)] $B$ (if nonzero) has an embedded  representative 
$C^B$ that is $J$-holomorphic for some $J\in \Jj(\Ss)$.
 \end{itemize}
\MS

%
%

\NI {\bf Step 1:} {\it If $\ell=1$ in \eqref{eq:dec1} 
then $A$ 
has an  embedded representative  
 that intersects $\Ss$ and the $C^{E_i}$ orthogonally.}

  \begin{proof}  We use the constructions and notations of  Example~\ref{ex:pole2}.
 Let us first suppose that $B=0$ and that there is a single curve $E_i$ in class $E$ so that
$A = S + mE$.  
    Then, with $a: = S\cdot E$, and $2\le k: = -S\cdot S \le 4$, we must have
$$
E\cdot A = a-m\ge 0,\quad S\cdot A = -k + ma\ge 0\;\;\Longrightarrow\;\;
 k\le ma \le a^2.
$$ 
If $m=1$ then we can construct the desired curve  as in the case $k\le 5$
in Example~\ref{ex:pole2}.  If $m\ge 2$ then  $a\ge 2$.  We take $\Si= S^2$, and $f:\Si\to S^2=C^E$ an $m$-fold cover branched at two of the intersection points of $C^E$ with $C^S$. 
Because $g(\Si) = 0$ we can put the poles of $\rho_\Si$ at the two branch points and hence 
can accommodate up to $2m\ge k$ poles.  

Now suppose that all coefficients $\ell_i$ in \eqref{eq:dec1} vanish.   Because $E_i\cdot E_j=0, i\ne j$ 
and $B\cdot E_i=0, \forall i$, we must have 
\begin{gather*}
a_i: = E_i\cdot  S\ge m_i,\quad B\ne 0\Longrightarrow h: = B\cdot S>0,\\ 
\quad 
 k\le \sum a_im_i +   h.
\end{gather*}
Since $k\leq 4$, if $\sum a_i+ h\geq 4$, the claim holds because one can use sections of the normal bundle to the $C^{E_i}$ and to $C^B$ with simple poles at each intersection point  with $C^S$ to accommodate the four poles of a section of the normal bundle to $C^S$. The claim is also true when $A=S+mE$, as we saw above.  Therefore, we need only consider the situation where $\sum a_i+ h \leq 3$ and either $B\neq 0$ (so that  $h\ge 1$) or there are at least two $E_i$.  This is possible only if all $a_i\le 2$.  But because $m_i\le a_i$ this means that again we only need consider two-fold covers.  Therefore the argument proceeds as before.

The general case, in which some coefficients $\ell_i$ are nonzero,  is similar.  Indeed, since the construction yields a representative that is orthogonal to the 
exceptional curves it makes no difference whether these lie in $\Ss$ or are other curves $C^{E_i}$. 
\end{proof}

\NI {\bf Step 2:} {\it Completion of the proof.} 
Suppose inductively that the results holds for all $\ell<\ell_0$ and consider $A$ with a decomposition~\eqref{eq:dec1}
with $\ell = \ell_0>1$.  We aim to show that there are nonnegative integers $\ell_j'\le \ell_j, m_i'\le m_i$ such that
$A': = S + \sum_{j\in \Ii_\Ee} \ell_j' S_j + \sum m_i' E_i + B$  
satisfies  condition (I) in \eqref{eq:dec1}.  
Then, because (II), (III) are automatically true, we may apply Step 1 to conclude that
 $A'$ has an embedded representative.  Therefore, the decomposition
$$
A = (\ell_0-1)S + \sum_{j\in \Ii_\Ee} (\ell_j-\ell_j') S_j  + \sum_i (m_i-m_i') E_i + A',
$$
also has the properties of \eqref{eq:dec1} but with  $\ell<\ell_0$.
Hence it has an embedded representative by the inductive hypothesis.

Therefore it remains to find suitable $\ell_j', m_i'$. For simplicity, let us first suppose 
that $\ell_i = 0$ for all $i$.  
As in Step 1, define
$
a_i: =E_i\cdot S$, and $h: = B\cdot S$ so that
$$
(*)\quad E_i\cdot A = \ell_0 a_i-m_i\ge 0,\quad (**)\quad S\cdot A = -\ell_0 k + \sum a_i m_i + h \ge 0.
$$
Here are some situations in which we can check that there is a  class $A'=S + \sum m_i' E_i + B$ 
that satisfies the numeric conditions (I).  
\begin{itemize}\item[(a)]
If $h + \sum_i a_i \ge 4\ge k$, then we may take $A' = S + \sum E_i + B$;  
\item[(b)] if all $m_i = 1$ we are in the previous case, and may take 
$A' = S + \sum E_i + B$;
\item[(c)] if $a_i = 1$ for all $i$ then $m_i\le \ell_0$ for all $i$  so that $\frac {m_i}{\ell_0}\le 1=a_i$ for all $i$, so that ($**$) gives $k\le \sum \frac {m_i}{\ell_0}+ \frac {h}{\ell_0} \le\sum_i a_i  + h$ and we are again in case (a);
\item[(d)] if there is $i$ with $m_i \ge 2$ and $a_i\ge 2$, then  $k\le 4\le 2a_i$ so that
we can take $A' = S + 2 E_i$;
\item[(e)] if there is only one curve $E_i$, then we may take $n: = \lceil \frac m{\ell_0}\rceil$, and $A' = S + nE + B$.
Note in this case that $E\cdot A' = a - n\ge 0$ since $a$ is an integer $\ge \frac m{\ell_0}$.
\end{itemize}
If none of these cases occur then there are at least two curves $E_1, E_2$ where $a_1>1, m_1=1$ and $m_2>1, a_2=1$.  
Further since $h + \sum_i a_i \le k-1$ we must have  $k=4$, $h=0$, $a_1=2$,  and no other $E_i$.  
But then $m_2\le \ell_0$ by ($*$) and $4\ell_0\le 2 + m_2$ by ($**$), which is impossible.
Hence in all cases there is a suitable class $A'$.  

Since the above argument is purely algebraic, it works equally well if some of the exceptional spheres
in \eqref{eq:dec1} lie in $\Ss$.
This completes the inductive step and hence the proof.
\end{proof}

\begin{rmk}\label{rmk:iv} \rm  By using
 the special $3$-fold cover in Remark~\ref{rmk:tori} one should be able to extend
this argument  to larger values of $k$.
\end{rmk}

\begin{cor}\label{cor:ii}  Proposition~\ref{prop:1}
holds for this $\Ss$.
\end{cor}
\begin{proof}  Under the given assumptions for the class $A$, Proposition~\ref{prop:Z} constructs a $1$-parameter family of $\Ss$-adapted nodal $A$-curves.   The above proof that amalgamates these into a single embedded  $A$-curve uses  patching procedures that are only slightly more complicated
than those in Proposition~\ref{prop:gp}.   Hence, as in Remark~\ref{rmk:t},  they may be carried out 
 for a $1$-parameter family, giving the required family of embedded curves. 
\end{proof}

\subsection{The asymptotic problem}\label{sec:asympt}

We now prove Theorem \ref{thm:asympt}  and Corollary \ref{cor:asympt}, using the patching procedures described in 
\S\ref{ss:geom}, as well as the inflation results explained in
\S\ref{s:inflat}. 

\begin{proof}[Proof of Theorem  \ref{thm:asympt}]
Here we assume that 
$(M,\om,\Ss,J)$ is a symplectic manifold with singular set $\Ss$, $J\in \Jj(\Ss)$ and $\Sigma^A$ is a nodal $J$-representative of some class $A\in H_2(M)$. 
If $\om$ is a rational class, a classical refinement of Donaldson's construction produces a symplectic curve $C^T$ for $T=\PD(N\om)$, $N\gg 1$, which intersects $\Ss,\Sigma^A$ transversally and positively \cite{D96}. The first statement of the theorem is a further refinement explained in \cite{O12}: whatever $\om$, for some $r\le {\rm rank\,} H^2(M)$, there is a decomposition 
$$
[\om]=\sum_{i=1}^r \beta_i\ \PD(T_i),\beta_i>0, 
$$
where $T_i$ are represented by embedded symplectic curves $C^{T_i}$, which again intersect $\Ss\cup \Sigma^A$ positively and transversally. At this point, we therefore have a $J$-nodal curve $\Ss\cup \Sigma^A\cup \Tt$ (where $\Tt=\cup C^{T_i}$), for some $J\in \Jj(\om)$. As will be explained in section \ref{s:inflat} (cf. Lemma~\ref{le:infl2}), we can perform a small inflation along $\Tt$ in order to get a symplectic form $\om' = \sum_i \beta_i'\ \PD(T_i)$ 
in a rational class, close to $\om$, still $J$-compatible.

On the other hand, Proposition~\ref{prop:infl1} guarantees the existence of a $J$-compatible symplectic form $\om'_\kappa$ 
 in class $\PD(A)+\frac{[\om']}{\kappa}$, for arbitrary large $\kappa$. 
Given the $\eps_i$, choose  $\ka\in \Q$ so that $\eps_i-\frac{\beta_i'}{\kappa} \ge 0$ and then choose $N_0$ so that
$N_0\eps_i,N_0\beta_i'/\kappa\in \Z$ for all $i$.
 Again by Donaldson's construction, for $k\gg 1$ there is an embedded curve $\Sigma$ 
that is  approximately $J$-holomorphic (hence $\om$-symplectic) and in class 
$$
[\Sigma]=k N_0\left(A+\frac{PD[\om']}{\kappa}\right)=kN_0\big(A+\sum \frac{\beta_i'}{\kappa}\ T_i\big),
$$
As before, $\Sigma$ can be required to intersect $\Ss,\Tt$ transversely and positively, meaning that $\Sigma$ is $J'$-holomorphic for some $J'\in \Jj(\Ss,\om)$. 
Then the given class $kN_0A_\eps:=kN_0(A+\sum\eps_i T_i)$  is  represented by the nodal curve 
$$
\Sigma\cup \bigcup kN_0(\eps_i-\frac{\beta_i'}{\kappa})C^{T_i}.
$$
 (Note that by construction each $N_0(\eps_i-\frac{\beta_i'}{\kappa})$ is a positive integer.)
Since $\Sigma$ has only transverse and positive intersections with $\Tt$, we can smooth this nodal curve to an embedded one
as in  Lemma~\ref{le:patching} (with $\ell=m=0$).
 \end{proof}
 
\begin{proof}[Proof of Corollary \ref{cor:asympt}]
Consider  the $k$-fold blow-up $\Hat {\CP}\,\!^2_k$ of $\CP^2$
 endowed with a symplectic form $\om$, a singular set $\Ss$, and a class $A=L-\sum \mu_iE_i$.  
 Slightly perturb $\om$ if necessary  so that  $[\om] =\ell-\sum \alpha_ie_i$  is rational.
 Since the union of closed balls
$\sqcup \ov {B}(\mu_i)$ embeds into $\CP^2$, there is a symplectic form in class $\ell-\sum \mu_ie_i$, and hence in nearby classes
$\ell-\sum (\mu_i +\de_i)e_i$.  
It follows that for sufficiently small $|\de_i|$, chosen so that $\mu_i + \de_i$ is rational, every integral class of the form
$A'=q(L-\sum (\mu_i+\de_i)E_i)\in H_2(\Hat {\CP}\,\!^2_k)$  where $q>0$, is reduced and has nonvanishing Gromov invariant.
 Applying Theorem \ref{thm:asympt} with $r=1$ and to such a class $A'$, we get an integral  class $T=\PD(N_0\om)$ and, for all positive $\eps\in \Q$, a symplectically embedded curve positively transverse to $\Ss$ in a class of the form
 $$
\begin{array}{rl}
N'(A'+\eps T)= & N'\big(L-\sum (\mu_i+\de_i)E_i+\eps N_0(L-\sum \alpha_iE_i)\big)\vspace{.1in}\\
 =& N'\Big((1+\eps N_0)L-\sum (\mu_i+\de_i+\eps N_0\alpha_i) E_i\Big)\vspace{.1in}\\
 =& \ds N'(1+\eps N_0)\Bigl( L-\sum \frac{\mu_i+\de_i+\eps N_0\alpha_i}{1+\eps N_0}E_i\Bigr). 
\end{array}
$$
Note that the choice of $N_0$ is independent of that of  $\de_i$ and $\eps$, though $N'$ depends on the latter choices.  
For sufficiently small (rational ) $\eps>0$ we may choose $\de_i: = \eps N_0(\mu_i-\al_i)$, so that the class 
$N'(A'+\eps T)$ is a multiple of $A = L-\sum \mu_i E_i$.
We conclude as claimed that for some $N$ the class $NA$ is represented by a $J$-curve for some $J\in \Jj(\Ss,\om)$.
\end{proof}

  \section{Symplectic inflation}\label{s:inflat}

We assume throughout this section that
$(M,\om)$ is a
blow up of a rational or ruled manifold so that the calculation of
$\Gr(A)$ is given by Lemma~\ref{le:SW}.
For short, we simply say that $M$ is rational/ruled.
We begin  in \S\ref{ss:inflat} by explaining the inflation process and proving Theorem~\ref{thm:inf} modulo some technical results.  Even in the absolute case, the details here are new:
we explain a streamlined version of the construction that is easy to generalize to the relative case.  
The  proofs of the technical results 
are deferred to \S\ref{ss:tech}.  In particular, Lemma~\ref{le:tech1} is a more detailed version of Lemma~\ref{le:infl00}.  


  \subsection{The main construction}\label{ss:inflat}

In this section we  work relative 
to a collection $\Cc$
of  surfaces $C^{T_j},1\le j\le L$,
 that may contain some or all of the components of $\Ss$ 
 and satisfies the following conditions.
\begin{cond}\label{c:Cc}\rm
\begin{itemize}\item[(a)]  
 Each $C^{T_j}$ is
$\om$-symplectically embedded, and lies in a class  $T_j$ with $T_j\cdot T_j = n_j\in \Z$;
\item[(b)]  
Each surface  $C^{T_j}$ is $\om$-orthogonal to all the components $C^{S_i}$ of $\Ss\less\Cc$ as well as to the other  
$C^{T_k}, k\ne j$. 
\end{itemize}
\end{cond}
In this situation we say that $\Cc$ is  {\bf $(\Ss,\om)$-adapted}, or simply {\bf $\Ss$-adapted},
and that the form $\om$ is {\bf $\Cc\cup\Ss$-adapted.}\footnote
{
This amounts to requiring that $\om$ satisfy the conditions in Definition~\ref{def:sing} with respect to the collection
$\Ss\cup \Cc$.  For we always assume that $\om$ is compatible with the given fibered structure near $\Ss$, and because  
of the orthogonality condition (b) we can always choose a compatible fibered structure near $\Cc$.} 
A component $C^{T_i}$ is called {\bf positive} (resp. {\bf negative)} if
$n_i\ge 0$ (resp. $n_i< 0$). We say that $\Cc$ is $J$-holomorphic
if the tangent space to each of its components is $J$-invariant. 
Similarly, we say that a nodal curve $\Sigma^A$ is {\bf $(\Ss\cup \Cc)$-adapted} if the collection of its components 
satisfies the above conditions with respect to $\Ss\cup\Cc$. 

 In applications, we will represent the class $A$ along which we want to inflate
by a nodal curve $\Si^A$  whose components give a decomposition
$A=\sum \ell_i S_i + \sum n_jB_j$ as in \eqref{eq:Adecomp0}, and then take $\Cc$
to contain the curves in the singular set $\Ss$ together with
suitable embedded representatives of the classes $B_j$ obtained via Proposition~\ref{prop:gp}.
Thus we can write $A = \sum m_j T_j$ for some integers $m_j\ge 0$, where $T_j$ are the classes of the components of $\Cc$.
Here, as always, we take the class $A$ to be integral.  However it is just as easy, and  convenient specially in the relative case, to inflate along classes $Y\in H_2(M;\R)$ of the form $Y: = \sum \la_i T_i$ where $\la_i\ge 0$ are real numbers. As will become clear, the important  point  is not whether $Y$ is integral but that the classes $T_i$ are represented by the submanifolds in $\Cc$.

We begin by stating a version of the basic inflation result.
(A simpler version was proved in \cite{Merr}
 using the pairwise sum as in \cite{LU}.)

\begin{prop}\label{prop:infl1} With $\Cc$ as above, 
let  $Y: = \sum_{i=1}^L \la_i T_i$ where $\la_i\ge 0$ and define $\la_{\max}: = \max_i \la_i$.
Then  there are constants  $\ka^0, \ka^1>0$, depending on $Y, \om $ and $\Cc$ 
and a smooth family of symplectic forms  $\om_{\ka,Y}, \ka \in [ -\ka^0,\ka^1],$ 
on $M$ such that the following holds for all $\ka$.
\begin{itemize}\item[(i)]   $[\om_{\ka,Y}] = [\om_0] + \ka \PD(Y)$, where $\PD(Y)$ denotes the Poincar\'e dual of $Y$.
\item[(ii)] $\om_{\ka,Y}$ is  
$(\Ss\cup \Cc)$-adapted.
 \item[(iii)]   If $Y\cdot T_j=0$ for some $j$
  the restrictions of $\om_{\ka,Y} $ and $ \om$ to $C^{T_j}$ are equal. 
\item[(iv)]  The constant $\ka^0$ depends on geometric information, namely $\om, \Cc$ and $\la_{\max}$, while  
 $\ka^1$ depends only on  $[\om]$, $\la_{\max}$, and the homology classes $T_i$.
 Moreover, if $Y\cdot T_i\ge 0$ for all $i$ then $\ka^1$ can be arbitrarily large.
\end{itemize}
\end{prop}

For short we will say these forms $\om_{\ka,Y}$ are constructed by {\bf $\Cc$-adapted inflation}. 
We will see in the proof (given in \S\ref{ss:tech}) that the curves along which we inflate are part of $\Cc$.

Note also that in this result we allow $\ka$ to be slightly negative.  We will call a deformation from $\om_0$ to $\om_{-\eps}$ a {\bf negative inflation}.  However, just as \lq\lq inflation" along a class $S$ with $S^2<0$ decreases $\om(S)$,
negative inflation along such a class
increases $\om(S)$.  
The next example shows why we cannot always take $\ka^1$ to be arbitrarily large.

\begin{example}\rm
If $T = E$ is the class of an exceptional divisor $C^E$, then 
 negative inflation along $Y = E$ by $-\ka$
changes $[\om]$ to $[\om] -\ka\,\PD(E)$, increasing the size of $C^E$ to $\om(E) + \ka$. On the other hand, positive inflation by $\ka$
to $[\om] +\ka\, \PD(E)$ decreases it to $\om(E) - \ka$ and so is possible only if $\ka<\om(E)$.
\end{example}

The same argument works in  $1$-parameter families.  More precisely, the following holds.

\begin{lemma} \label{le:infl2} Let $\om_t, t\in [0,1],$ be a smooth family of symplectic forms on $M$ and  $\Cc_t,  t\in [0,1],$ be a smooth family of $(\Ss, \om_t)$-adapted submanifolds in the classes 
$T_i, 1\le i\le L$.  
Let
 $Y_t: = \sum_{i=1}^L \la_i(t) T_i$ with $\la_i(t)\ge 0$.
Then the following holds.
\begin{itemize}\item[(i)]  There are  constants $\ka^0,\ka^1>  0$ 
and
a  $2$-parameter family of symplectic forms  $\om_{t,\ka,Y}, t\in [0,1], -\ka^0\le \ka\le \ka^1$ that for each $t$ satisfies the conditions (i) -- (iv) of Proposition~\ref{prop:infl1} with respect to $\Cc_t$ and $Y_t$. 
In particular, $[\om_{t,\ka,Y}] = [\om_t] + \ka \PD(Y_t)$ for all $t\in [0,1], \ka\in [-\ka^0,\ka^1]$.
\item[(ii)]  One can construct this family $\om_{t,\ka,Y}, t\in [0,1],
-\ka^0\le \ka\le \ka^1,$ so that  it extends any given paths for $t = 0,1$
that are constructed by $\Cc_0$- (or $\Cc_1$-) adapted inflation.
\end{itemize}
\end{lemma}

In order to apply Lemma~\ref{le:infl2} to prove Theorem~\ref{thm:inf} we need first  to find suitable classes
$A$  along which to inflate, and then construct the families $\Cc_t$.
The following argument
that deals with the case $\Ss=\emptyset$ is adapted from
 \cite{Mcd}. For simplicity, we explain it only when $M$ is a blow up of $\C P^2$.

\begin{lemma}  \label{le:infl3}
Let $M$ be a blow up of $\C P^2$, and
suppose given a smooth family of symplectic forms $\om_t, t\in [0,1],$ on $M$ with $[\om_0] = [\om_1]$.
Then there is a family of symplectic forms $\om_{st}, 0\le s,t\le 1,$ such that
$$
\om_{s0} = \om_0\mbox{ and } \om_{s1} = \om_1, \forall s,   \quad 
\om_{0t} = \om_t \mbox{ and }  [ \om_{1t}] = [\om_0] \ \forall t.
$$
\end{lemma}

\begin{proof}  Write $L, E_j, j=1,\dots,K,$ for the homology classes of the line and
the obvious exceptional divisors, and then define $\ell: = \PD(L), e_j = \PD(E_j)$ so that $e_j(E_j) = -1$. 
\MS

\NI {\bf Case 1:} {\it  $[\om_0]$ is rational.}

We claim that for sufficiently large integer $N$
the following conditions hold, where
$\Pp^+$ is the positive cone as in
Fact~\ref{f:*}:
\begin{itemize}\item
$N[\om_0] \pm e_j \in \Pp^+$
 for all $j$,
 \item  the class $A_j^\pm = \PD(N[\om_0] \pm e_j)$ is reduced, i.e. $A_j^\pm\cdot E\ge 0$ for all
 $E\in \Ee$ .
 \end{itemize}
By the openness of the space of symplectic forms, there is an $\eps_0>0$ such that the classes $[\om_0]+\eps e_j$ 
have symplectic representatives $\om_\eps$ for all $|\eps|<\eps_0$. Taking $N<\nf{1}{\eps_0}$, the first claim obviously holds. Moreover, $N[\om_0]\pm e_j$ must evaluate positively on each exceptional class $E\in \Ee(\om_{\pm \nf{1}{N}})$. The claim follows  by deformation invariance of  $\Ee$.

Next, with $A^\pm_j = \PD(N[\om_0] \pm e_j)$, Corollary~\ref{cor:SW} implies that $\Gr(qA^\pm_j )\ne 0$ for sufficiently large $q$. 
 It follows 
that for any 
deformation
$\si_t$, given a
 generic $1$-parameter family $J_t$ of $\si_t$-tame almost complex structures, there is (after possible reparametrization
  with respect to $t$) a family of  embedded connected $J_t$-holomorphic submanifolds $C^\pm_{t,j}$
 in class $qA_j^\pm$. If we do this for each of the classes $A^+_1,A^-_1,A^+_2,\dots$ in turn, possibly reparametrizing at each step, we may suppose that
  there is a family $\Cc_t^J$ of
 $J_t$-holomorphic submanifolds in these classes. We can finally  perturb them
to get a family 
$\Cc_t^\Aa, t\in [0,1],$ composed of $\om_t$-orthogonally intersecting curves for each $t$.
Observe that the homological intersections $A_j^{\pm}\cdot A_i^{\pm}$ are all nonnegative when $i\neq j$ (as well as for $A_i^+\cdot A_i^-$) because the classes are $J$-represented; also  $({A_i^\pm})^2>0$ by hypothesis ($A_i^\pm\in \Pp^+$). Hence every class $\sum \la_j A_j^\pm$, with $\la_j\ge 0$
intersects every component of $\Cc_t^\Aa$ positively  for all $t$, and so can be used for 
arbitrary positive inflations by Lemma~\ref{le:infl2}.

The family $\om_{st}$ is constructed in three stages.  The first stage for $s\in [0,s_1]$ implements the reparametrization.  
The second stage is the inflation. 

Each class $[\om_t]$ has a unique decomposition as
$$
[\om_t]=c(t)[\om_0]+\sum_{j\in \Ii^+(t)} \lambda_j(t)e_j -\sum_{j\in \Ii^-(t)}\lambda_j(t)e_j, \hspace{,5cm} c(t),\lambda_j(t)>0,
$$
where $\Ii^+(t), \Ii^-(t)$  are suitable disjoint subsets of $\{1,\dots, K\}$ for each $t$, and  the functions
$c(t),\lambda_j(t)$ are smooth.
Define the class 
$$
Y_t:=\frac 1{s_2-s_1} \Bigl(\sum_{j\in \Ii^+(t)} \lambda_j(t)A_j^- +\sum_{j\in \Ii^-(t)}\lambda_j(t)A_j^+\Bigr).
$$
Note that here we  pair $j\in \Ii^+(t)$ with  $A^-_j = \PD(N[\om_0]-e_j)$. It follows that
inflation along $Y_t$ gives a smooth family of symplectic forms $\om_{st}$, $s\in[s_1,s_2]$, in class 
\begin{eqnarray*}
[\om_{st}]&=& \ds[\om_t]+\frac{s-s_1}{s_2-s_1} \Bigl( \sum_{ \Ii^+(t)} \lambda_j(t)\PD(A_j^-) +\sum_{\Ii^-(t)}\lambda_j(t)\PD(A_j^+) \Bigr)\hspace{.08in}\\
&=& \ds \Bigl(1+N\frac{s-s_1}{s_2-s_1}\sum \lambda_j(t)\Bigr)[\om_0] +\sum_{\Ii^+(t)} \Bigl(1-\frac{s-s_1}{s_2-s_1}\Bigr)\lambda_j(t)e_j\\
&& \qquad\qquad  -\sum_{\Ii^-(t)} \Bigl(1-\frac{s-s_1}{s_2-s_1}\Bigr) \lambda_j(t) e_j.
\end{eqnarray*}
For $s=s_2$, the classes $[\om_{s_2t}]$ are proportional to $[\om_0]$. The third stage consists of a rescaling, which gives $[\om_{1t}]=[\om_0]$. Observe that $Y_0=Y_1=0$ so that this whole process does not modify $\om_0$ and $\om_1$. \MS

\NI {\bf Case 2:} {\it  $[\om_0]$ is irrational.} 

In this simple situation ($\Ss=\emptyset$), it is well-known that the \lq\lq deformation to isotopy" statement 
is equivalent to the claim that the space of symplectic embeddings of disjoint closed balls of a fixed size into $\CP^2$ is path connected.  But if this holds for  balls of rational size, it is obviously also true for
balls of irrational size since we can always extend an embedding of irrational balls to
 slightly larger rational balls, isotop this as required, and then 
restrict the isotopy to the original balls.
We now give the formal proof  that keeps track of this argument, because it will adapt to the situation $\Ss\neq \emptyset$. 

 Rescale so that $\om_0(L) = 1$
and   write
$
[\om_0] = \PD(L) - \sum_j \la_j e_j$, where $\la_j>0$.
\footnote{This is the only step in the argument that fails when $M$ is ruled.
In this case, one should replace $L$ by the class of some section of the ruling that has nontrivial Gromov invariant, and add the class of the fiber $F$ (which is always represented) to the exceptional classes.}
For $t=0,1$ choose a generic  $\om_t$-tame almost complex structure $J_t$. Then 
there are disjoint embedded $J_t$-holomorphic curves $C^{E_j}_t$ in the classes $E_j, 1\le j \le K,$ for $t=0,1$.
Choose $\ka^0>0$ so that we can negatively  inflate 
along these curves  for $t=0,1$ and for $-\ka^0\le \ka \le 0$,
 and then choose rational numbers $\mu_j=\la_j+\delta_j$ with $\delta_j<\ka^0$.
Then, by negatively inflating along
the curves $C^{E_j}_0, C^{E_j}_1$,
construct  families of forms $\om_t$ for $t\in [-1,0]$ and $t\in [1,2]$
so that  $[\om_{-1}]=[\om_2]$ is rational:
$$
[\om_{-1}] =[\om_2] = \PD(L) - \sum_j \mu_j e_j, \quad \mu_j\in \Q.
$$
Because the endpoints of the path $\om_t, -1\le t\le 2$ are
now equal and rational, as before we may
homotop this  deformation to an isotopy $\rho_t, t\in [-1,2],$ with $\rho_t = \om_t$ at $t=-1,2$.   
Note that the set of classes $A_j^\pm = N[\om_{-1}]\pm e_j$ along which we must now inflate depends on
$[\om_{-1}]$.    Hence the family $\Cc_t^\Aa, t\in [-1,2],$  does as well.  
Further, because the classes $E_j= \PD(e_j)\in \Ee$ are represented by unique $J_t$-holomorphic
embedded spheres for all  generic $1$-parameter path $J_t$, we may simply add representatives of the classes $E_j$ to the family $\Cc_t^J$, and then straighten out the components of the curves in $\Cc_t^J$ 
using 
Lemma~\ref{le:gp1} to obtain a  family $(\Cc_t^\Aa)',t\in [-1,2]$
of curves with pairwise orthogonal intersections,
 that contains embedded representatives $C^{E_j}_t$ of each class $E_j$
as well as 
the components of $\Cc^\Aa_t$.
Moreover, we can suppose at $t=-1,2$ that these curves equal the previously chosen ones at $t=0,1$ respectively.
Then by Lemma~\ref{le:infl2}, the isotopy $\rho_t, t\in [-1,2],$ consists of forms that are 
nondegenerate on the $C^{E_j}_t$.

More precisely, in the three stages defined above, we get for some $0<s_3<1$
 a $2$-dimensional family $\om_{st}$ of symplectic forms, $t\in [-1,2], s\in[0,s_3]$ 
that homotops $\om_t$ (for $s=0$) to $\rho_t=\om_{s_3t}$, 
where $[\om_{s_3t}]\equiv [\om_{-1}]=[\om_0]-\sum \delta_je_j$. By construction, the curves $C^{E_j}_t$ are $\om_{s_3t}$-symplectic, with area larger than $\delta_j$, so  the last stage consists in performing a positive inflation of size $\delta_j$ along the curves $C^{E_j}_t$,  and a reparametrization in $t$,  
 in order to straighten $\om_{s_3t}, t\in [-1,2], $ to $\om_{1t}, t\in [0,1],$ in class $[\om_0]$.    Note that at the endpoints this last step reverses the original negative inflation of $\om_0$ to $ \om_{-1}$ and $\om_1$ to $\om_2$.
Therefore the final isotopy $\om_{1t},t\in [0,1],$ starts at $\om_0$ and ends at $\om_1$, as required.  
\end{proof}

In order to carry out this proof in the  case of isotopies relative to $\Ss$, one needs to find suitable
representatives of all the classes involved in the above proof, the $A_j^\pm$ when
$[\om_0]$ is rational, and also suitable substitutes for the $E_j$ in the general case.
In order to deal with the latter we will need to work relative to a smooth 
$\Ss$-adapted family 
that for each $t$  contains 
 representatives of the classes corresponding to the $E_j$.
Here is the main result about the existence of such representatives. 
Note also that the condition on $d(A_j)$ comes from Lemma~\ref{le:SW}, and is needed to ensure some Gromov invariant does not vanish.

\begin{prop}\label{prop:Zt}
Let $\om_t, t\in [0,1]$, be a path of symplectic forms as in Theorem~\ref{thm:inf}, and
$\Cc_t$ be a smooth $(\Ss,\om_t)$-adapted family of surfaces in the classes $T_1,\dots,T_L$. 
Suppose 
given a finite set $\Aa=\{A_1,\dots,A_K\}$  of 
$\Ss$-good
classes 
such that 
\begin{itemize}
\item[-] $A_i\cdot T_j\ge 0$ for all $1\le j\le L$, and
\item[-] $d(A_j)>0$ for all $j$.  Moreover, $d(A_j)\ge g+\frac k4$, if $M$ is the $k$-point blow up of a ruled surface  of genus $g$. 
\end{itemize}
Choose a smooth path $J_t\in \Jj(\Ss,\om_t,\Aa), t\in [0,1]$ of $(\Cc_t,\om_t)$-adapted almost complex structures.
Then, possibly after reparametrization with respect to $t$, the path
 $(J_t)_{t\in [0,1]}$ can be perturbed to a smooth 
 $(\Cc_t,\om_t)$-adapted path
 $
 \bigl(J_t'\in \Jj_{semi}(\Ss,\om_t,\Aa)\bigr)_{t\in [0,1]}$
  such that
 for each $1\le j\le K$ there is   a smooth family $\Si_t^{A_j}, t\in[0,1],$  of $J_t'$-holomorphic and 
$(\Ss\cup\Cc_t,\om_t)$-adapted   nodal curves in class $A_j$.  Moreover the corresponding decompositions 
$$
A_j=\sum \ell_{ji} S_i +\sum m_{ji} E_{ji}+ B_j, \hspace{.5cm} E_{ji}^2=-1,\
$$
of \eqref{eq:Abir} have the property that $\Gr(B_j)\ne  0$ for all $j$.
\end{prop}

\begin{proof}  The proof when there is only one class $A$ and when $\Cc_t = \Ss$ is essentially the same as that of Proposition~\ref{prop:Z}.
The argument works just as well if $\Cc_t$ is strictly larger than $\Ss$. 
Since by hypothesis  $d(B)\ge d(A)>0$, we can always choose the 
set of  $k: = \frac 12 d(B)$ points so that at least one does not lie in the three-dimensional set 
$\cup_t \Cc_t$.  Hence 
we are free to perturb $J_t'$ near some point on the $B$-curve which means that the genericity
 arguments work as before.   

Finally, if $N>1$ we argue 
 by induction on $N$.  Note that at each stage we may have to reparametrize.
 Further, to finish  the $i$th stage we should  apply
  the straightening argument in 
  Lemma~\ref{le:gp1}  to  make the components of the $A_i$-nodal curve orthogonal to $\Cc_t$ and all
  components for the previously constructed nodal curves $\Si_t^{A_j}, j<i$. Then at the $(i+1)$st stage, we repeat the argument with this enlarged family $\Cc_t'$.
\end{proof}

\begin{proof}[Proof of Theorem~\ref{thm:inf}]  Recall the statement: $M$ is a blow-up of $\CP^2$ or a ruled surface, we have a family of symplectic forms $\om_t$, $t\in[0,1]$ with $[\om_0]=[\om_1]$ and, as in the previous lemma, we want to find a homotopy of symplectic forms $\om_{st}$ between $\om_t$ (for $s=0$) and  an isotopy $\om_{1t}$ (meaning that $[\om_{1t}]$ is constant) with fixed ends: $\om_{s0}=\om_0$, $\om_{s1}=\om_1$ for all $s$. This time, the situation is {\it relative to $\Ss$}, meaning that we assume that the forms $\om_t$ are nondegenerate  on $\Ss$, and we want our homotopy $\om_{st}$ to 
have the same property. 
\MS

\NI {\bf Case 1:} {\it $[\om]$ is rational}

 In order to adapt  the proof of Lemma~\ref{le:infl3} we first choose an analog
of the basis $L,E_j$ for $H_2(M)$.  
We take $L$ to be any class with nonzero Gromov invariant, so that $\om_t(L)>0$ for all $t$ and 
then choose integral classes $d_1,\dots,d_K$  that together with $\PD(L)$ form a basis of $H^2(M;\Q)$.
Define the classes $A^{\pm}_j:=\PD(N\om_0\pm d_j)$ as before,
using the openness of the space of symplectic forms to find a suitable value of $N$ for which  these classes 
are all $\Ss$-good and also satisfy the
 enhanced condition on $d(A_j)$ when $M$ is ruled.  
This is possible by Corollary~\ref{cor:SW}.
We then use Proposition~\ref{prop:Zt} with $\Cc_t = \Ss$ and $\Aa = \{A^{\pm}_j: 1\le j\le K\}$ to get a smooth family of $(\Ss, \om_t)$-adapted nodal curves $\Si_{t,j}^\pm$ in classes $A_j^\pm$.  
Straighten out their components using 
Lemma~\ref{le:gp1}
 to obtain an $(\Ss,\om_t)$-adapted family $\Cc_t^\Aa$ that contains $\Ss$.

We next  claim that each class $A_j^\pm$ has nonnegative intersection with the classes of the components of $\Cc_t^\Aa$. 
 To see this, consider the  decomposition 
$$
A_j^\pm=\sum \ell_{ji}^\pm \ S_i +\sum m_{ji}^\pm\ E_{ji}+ B_j^\pm, \hspace{.5cm} E_{ji}^2=-1,\
\Gr(B_j^\pm)\neq 0
$$ 
associated
via Proposition~\ref{prop:Zt}
 to the nodal curves $\Si_{t,j}^\pm$.  We chose the  classes $A_j^\pm$ to 
be $\Ss$-good.
Therefore they have nonnegative intersection with the components of $\Ss$ as well as all exceptional classes $E_j^{\pm}$.  Further they have nonnegative intersection with the  $B_j^\pm$ 
because both the $A_j^\pm$ and $B_k^\pm$ have nontrivial Gromov invariant and hence are represented by embedded curves for generic $J$.
Hence
 Lemma \ref{le:infl2} allows inflation along any nonnegative linear combination of the $A_j^\pm$, and these inflations provide symplectic forms which are nondegenerate on $\Ss$.

The family $\om_{st}$ is then constructed in the same three stages as in the previous proof: reparametrization, inflation along the classes 
$$
Y_t=\sum_{\Ii^+(t)} \lambda_j(t) A_j^-+\sum_{\Ii^-(t)} \lambda_j(t) A_j^+
$$
(where $[\om_t]=c(t)[\om_0]+\sum_{\Ii^+} \lambda_j(t)d_j-\sum_{\Ii^-} \lambda_j(t)d_j$), and rescaling. The result at $s=1$ is an isotopy $\om_{1t}, t\in [0,1],$ consisting of symplectic forms that
restrict on $\Ss$ to a possibly varying family of forms that are $\Ss$-adapted
and all lie in the same cohomology class.\footnote{Note that we cannot invoke part (iii) of 
Proposition~\ref{prop:infl1}  to claim that the forms are constant on $\Ss$ throughout the deformation because some of the classes $Y_t$ might have nontrivial intersection with $\Ss$.}  

Finally, if $\om = \om'$ near $\Ss$ then
$\om_{10} = \om_{11} = \om$ near $\Ss$ by construction, and we can arrange
that the final isotopy $\om_{1t}$ is constant near $\Ss$ by an easy application of a  Moser's type argument.
Details are left to the reader.

\MS

\NI {\bf Case 2:} {\it $[\om]$ is irrational}

When $[\om]$ is irrational, we reduce to the rational case by first doing a small
 \lq\lq negative inflation" along suitable classes,
$F_1,\dots,F_K$, where 
$K=k+1$ or $k+2$ 
depending on whether $M$ is a $k$-fold blow-up of $\CP^2$ or of a ruled surface.   
These classes are 
obtained as follows. 
Choose integral classes 
$a_1,\dots, a_K$ that are multiples of classes  close to 
$[\om]= [\om']$,  so that
$$
[\om] = 
\sum_{i=1}^K \mu_i a_i,
$$
for some  $\mu_i\in 
\R^+$.
By the openness of the space of symplectic forms, we may assume that the classes $a_i$ have symplectic representatives and take positive values on the components 
$S_i$ of $\Ss$.    Then,  the classes  $F_i: = \PD(a_i)$ satisfy all the conditions 
needed to be $\Ss$-good
except 
that $\Gr(F_i)$ could vanish. Therefore, by replacing the $a_i$  by suitable multiples as in Corollary~\ref{cor:SW}, we can assume that each $F_i: = \PD(a_i)$ 
is $\Ss$-good,
and, if relevant, has $d(F_i)\ge g+\frac k4$ as in   Proposition~\ref{prop:Zt}.
By applying this proposition with $\Cc_t = \Ss$, we can find a smooth 
 path $
 \bigl(J_t'\in \Jj_{semi}(\Ss,\om_t,\Aa)\bigr)_{t\in [0,1]}$
  such that
 for each $1\le i\le K$ there is   a smooth family $\Si_t^{F_i}, t\in[0,1],$  of $J_t'$-holomorphic and 
$(\Ss,\om_t)$-adapted   nodal curves in class $F_i$.  
 Straighten out their components using  
 Lemma~\ref{le:gp1}
 to obtain an $(\Ss,\om_t)$-adapted family $\Cc_t^\Ff$.

As in the proof of Case 1,  each class $F_i$ has nonnegative intersection 
with the classes of the components of $\Cc_t^\Ff$. Hence 
Proposition \ref{prop:infl1}
allows negative 
$\Cc_0^\Ff$ (resp. $\Cc_1^\Ff$)-adapted inflation along   
any class $Y_{\bf \mu}:=\sum \mu_iF_i$, $\mu_i\in[0,1]$ by $-\kappa$ for all $\kappa$ less than some $\kappa_0$ (recall that $\kappa_0$ depends only on $\mu_{\max},\om,\Ff$, but not on the class $Y$ itself). Stated differently, $\Cc^\Ff_t$-adapted negative inflation along classes $\sum \mu_iF_i$, $\mu_i\in[0,\kappa_0]$ are possible for all $\kappa<1$. 

Now choose small constants $\de_i \in [0,\ka^0[$ so that 
$$
[\om]_\de = 
 \sum_{i=1}^K (\mu_i - \de_i) a_i,
$$
is rational.  
Define $\om_{t}, t\in [-1,0],$ (resp. $t\in [1,2]$) to be the family of forms  obtained from $\om= \om_0$ (resp.
$\om'= \om_1$) by negative $\Ss\cup\Cc_0^\Ff$- (resp. $\Ss\cup\Cc_1^\Ff$-) adapted inflation in class $Y_F: = \sum \de_i F_i$.

Then $[\om_{-1}] = [\om_2] = [\om]_\de$ is rational.  Hence we
may  apply the argument of Case 1 to the extended deformation $\om_t, t\in [-1,2],$ 
that has rational  and cohomologous endpoints.  The only new point is that we 
 construct the nodal curves $\Si^\pm_{t,j}$ in classes $A^\pm_j$ to be 
$\Ss\cup\Cc_t^\Ff$-adapted 
rather than $\Ss$-adapted.
This means that, in the notation of the proof of Lemma~\ref{le:infl3}, the isotopy $\rho_t, t\in [-1,2],$ 
from $\om_{-1}$ to $\om_2$ consists of forms that are $\Ss\cup\Cc_t^\Ff$-adapted.    Hence this isotopy can be positively inflated by a
$\Cc_t^\Ff$-adapted inflation in class $Y_F: = \sum \de_i F_i$ to an isotopy that joins the original form $\om$ to $\om'$.  This completes the proof. 
\end{proof}

\subsection{Proof of technical results}\label{ss:tech}

It remains to prove  Propositions~\ref{prop:infl1}  and Lemma~\ref{le:infl2}.
These use entirely soft methods.

Before embarking on the details of the proof of Proposition~\ref{prop:infl1}, we recall the basic inflation process; cf. \cite{Mcd,LU,B}.
Given  a symplectically embedded surface  $C$ with $C\cdot C = n \in \Z$ we normalize $\om$ in some neighborhood $\Nn$ of  $C$ as follows.
If  $r$ is a radial coordinate in the bundle $\pi:\Nn\to C$ where
$\Nn=\{\frac{r^2}2< \eps\} $, we write
$$
\om = \pi^*(\om|_C) + \tfrac 12 d(r^2\al),\quad r\in [0,\sqrt{2\eps}),
$$
where  $\al$ is a connection $1$-form with
$$
d\al = -\tfrac n{\om(C)}\pi^*(\om|_C).
$$
We then choose a nonincreasing compactly supported function $f:[0,\eps)\to [0,1]$ that is $1$ near $s=0$, and define
$$
\rho: = -d(f( \tfrac{r^2}2) \al).
$$
Consider the family of forms
\begin{eqnarray}\label{eq:nondeg}
\om + \ka \rho: &=& \pi^*(\om|_C) +  \tfrac 12  d(r^2\al) -{\ka}\ d(f( \tfrac{r^2}2) \al)\\ \notag
&=& \left( 1 + \tfrac{n}{\om(C)} \bigl( \ka f( \tfrac{r^2}2) - \tfrac {r^2}2 \bigr)\right) \pi^*(\om|_C) +  \bigl(1 + \ka |f'|\bigr) r dr\wedge \al.
\end{eqnarray}
By construction, this form lies in the class $[\om] +\ka \PD(C)$.
If $n\ge 0$ it is nondegenerate for all $\ka\ge 0$, and is also nondegenerate in some interval $-\ka^0 \le \ka
<0$, where the bounds on $\ka^0$   come from both terms: in particular, because we need $1 + \ka |f'|>0$
the bound $\ka^0$ depends on the size of $\eps$ and hence of the neighborhood $\Nn$.
If $n<0$  the first term also presents a significant obstruction, and we can only inflate for $\ka<\ka^1$ where
$|n|\ka^1< \om(C)$.

In the situation of 
Lemma \ref{le:infl2},
 we assume that $\Cc_t, t\in [0,1]$, is a smooth family of symplectic
submanifolds satisfying Condition~\ref{c:Cc} with respect to the forms $\om_t$,
with an associated family of local fibered structures $\Ff_t$ on a
neighborhood $\Nn(\Cc_t)$ as described just after Definition~\ref{def:sing}.
In particular,
each intersection point $q_t$ of $C^{T_i}_t$ with $C^{T_j}_t$ has a neighborhood $
\Nn_{q_t} $, which is a connected component of $ \Nn(C^{T_i}_t)\cap \Nn(C^{T_j}_t)
$ with product
structure given by the projections to $C^{T_i}_t$ and $C^{T_j}_t$.  We fix
corresponding polar coordinates $r_{t,i}, \theta_{t,i}, r_{t,j},\theta_{t,j}$ in the fibers of $\Ll_{t,i}$ and $\Ll_{t,j}$
that vary smoothly with $t$.
We assume that these neighborhoods $\Nn_{q_t}$ have disjoint closures for $q_t\in
\cup_{i\ne j} ( C^{T_i}_t\cap C^{T_j}_t)$, and
 then extend each radial function $r_{t,i}$ smoothly over $\Nn(C^{T_i}_t)$.
(This amounts to choosing a restriction of the structural group of $\Ll_{t,i}$ to $S^1$.)
We assume that for suitable constants $\eps_i>0$
\begin{equation}\label{eq:omC00}
\Nn(C^{T_i}_t) = \{x\in \Ll_i: r_{t,i}(x)\le  \sqrt{2\eps_i}\}.
\end{equation}
We also define
\begin{equation}\label{eq:omC0}
    \om^{T_i}_t: =   \om_t\big |_{C^{T_i}_t}.
\end{equation}
Finally, we shrink
 the neighborhoods as necessary  so that for negative $C^{T_i}_t$
we have
\begin{equation}\label{eq:Nn}
\sum_{q_t\in C^{T_i}_t\cap C^{T_j}_t, j\ne i}\ \int_{C^{T_i}_t\cap \Nn_{q_t}} \om^{T_i}_t \;\le \;\tfrac 12
\int_{C^{T_i}_t}\om^{T_i}_t.
\end{equation}

\begin{lemma}\label{le:tech1}  For each $C^{T_i}_t\in \Cc_t$ there are constants  
$\ka_i^0, \ka_i^1>0$
and a family of forms $\rho_{t,i}$ with the following properties:
\begin{itemize} \item $[\rho_{t,i}]=\PD(T_i)$;
\item $\rho_{t,i}$ is supported in the fibered neighborhood $\Nn(C^{T_i}_t)$ for each $t,i$;
\item $\rho_{t,i}$ is compatible with the product structure, namely of the form $d f_{t,i}(\frac{r_{t,i}^2}{2})\wedge d\theta_{t,i}$, on the product \nbd s $\Nn_p$ of each  $p\in C^{T_i}_t\cap \bigl(\Cc\less C^{T_i}_t\bigr)$;
\item  $\om_t + \ka \rho_{t,i}$ is symplectic  and $\Ss\cup\Cc_t$-adapted for $-\ka^0_i \le \ka<\ka_i^1$ and all $t$.
\end{itemize}
Moreover, $\ka_i^1$ can be arbitrarily large if $n_i:=T_i\cdot T_i\ge 0$ and otherwise
depends only on cohomological data, namely $n_i$ and $\om_t(T_i): = \int_{C^{T_i}_t} \om^T_{t,i}$.
Moreover, it is an increasing function of $\om_t(T_i)$.
\end{lemma}
\begin{proof}
{\bf Step 1:}\  {\it We may assume that there are connection $1$-forms $\al_{t,i}$ on the bundles $\pi_{t,i}: \Ll_{t,i}\to C^{T_i}_t$ such that
\begin{equation}\label{eq:omC1}
\om_t\big |_{ \Nn(C^{T_i}_t)} = \pi_{t,i}^*(\om^{T_i}_t) +\tfrac 12 d(r_{t,i}^2 \al_{t,i}), \quad 1\le i\le L,
\end{equation}
  where $r_{t,i}$ is the radial coordinate in the fiber of $\Ll_{t,i}$ described above.}

After possibly shrinking the neighborhoods $\Nn(C^{T_i}_t)$,   this can be achieved  by a standard Moser type argument.$\hfill\Box$
\MS

Next, denote by $g_{t,i}:C^{T_i}_t\to \R$ the curvature function of $\al_{t,i}$: thus
\begin{equation}\label{eq:omC2}
 d \al_{t,i} = -n_i \pi_i^*(g_{t,i} \om^T_{t,i}), \quad 1\le i\le N.
\end{equation}
Note that $g_{t,i} = 0$ in each product neighborhood because $\om_t$ is a product there.
\MS

\NI {\bf Step 2:}\  {\it  We may assume that
 $g_{t,i}(x)\ge 0$ for all $x\in C^{T_i}$,
    and satisfy the following pointwise upper bound on the negative curves (those with $n_i<0$):
\begin{equation}\label{eq:omC4}
 g_{t,i}(x) \le \frac 2 {\om_t(T_i)}.
\end{equation}
 }
 Again this follows by a standard Moser argument.  Note that to achieve this bound
 we must use condition~\eqref{eq:Nn} because $\int_{C^{T_i}_t}d\al_{t,i} = -n_i$ is fixed, while
 $g_{t,i} = 0$ in  each product neighborhood.$\hfill\Box$
 \MS

\NI
{\bf Step 3:}\  {\it Completion of the proof.}

Choose a family of smooth compactly supported functions  $f_{t,i}:
[0,\sqrt{2\eps_i})\to \R$  (where
$\eps_i$ is as in \eqref{eq:omC00})
that equal $1$ near $r=0$.
With $\rho_{i,t}: = d\big(f_{t,i}(\frac{r^2}2)\  \al_{t,i}\big)$, we have
\begin{eqnarray*} 
\om_t + \ka \rho_{i,t} &=& \pi_i^*(\om^T_{t,i}) +  \tfrac 12  d(r^2\al_{t,i}) -{\ka} d(f_{t,i}( \tfrac{r^2}2) \al_{t,i})\\ \notag
&=& \left( 1 + n_i\pi_i^*(g_{t,i}) \bigl( \ka f_{t,i}( \tfrac{r^2}2) - \tfrac {r^2}2\bigr)\right) \pi_i^*(\om^T_{t,i}) +  \bigl(1 + \ka |f'_{t,i}|\bigr) r dr\wedge \al_{t,i}.
 \end{eqnarray*}
As before, when $n_i\ge 0$ these forms are nondegenerate for all $\ka\ge 0$  and for 
$\ka>-\ka_i^0$, where $\ka_i^0$ depends only on the size $\eps_i$ of $\Nn_{t,i}$.
When $n_i<0$ we have similar limits for $\ka^0$, but now must only consider $\ka<\ka_i^1$, where the size of $\ka_i^1$
is determined by the requirement that the form $\om_t + \ka \rho_{i,t} $ restrict positively to $C^{T_i}_t
 = \{r=0\}$.
Since $f_{t,i}(0) = 1$ and $g_{t,i}$ satisfies \eqref{eq:omC4} this depends only on $n_i$ and $\om_t(T_i)$.
The other properties of these forms are clear.
\end{proof}

\begin{proof}[Proof of  Proposition~\ref{prop:infl1} ]
This proposition states  the following.

\begin{itemlist}\item{}{\it
Let 
 $Y: = \sum_{i=1}^L \la_i T_i$ where $\la_i\ge 0$.
Then  there are constants $\ka^0, \ka^1>0$, depending on $Y$ and $\Cc$ and 
 a smooth family of symplectic forms $\om_{\ka,Y}$ on $M$ such that the following holds for all $\ka\in [-\ka^0,\ka^1]$.
\begin{itemize}\item[(i)]   $[\om_{\ka,Y}] = [\om_0] + \ka \PD(Y)$, where $\PD(Y)$ denotes the Poincar\'e dual of $Y$.
\item[(ii)]
 $\om_{\ka,Y}$ is  $\Ss\cup\Cc$-adapted.
\item[(iii)] If $Y\cdot T_j=0$ for some $j$
  the restrictions of $\om_{\ka,Y} $ and $ \om$ to a neighborhood of $C^{T_j}$ are equal.
  \item[(iv)]  The constant $\ka^0$ depends on geometric information, namely $\om, \Cc$ and 
  $\la_{\max}$,
  while  $\ka^1$ depends only on  $[\om]$ and the homology classes $T_i, Y$.
 Moreover, if $Y\cdot T_i\ge 0$ for all $i$ then $\ka^1$ can be arbitrarily large.
\end{itemize}
}
\end{itemlist}

We use the notation of Lemma~\ref{le:tech1} omitting $t$ since for the moment we are considering single forms.
First consider the forms
  $$
\om_{\ka,Y}': = \om +  \sum_{i=1}^L \la_i \ka \rho_i. 
  $$
Because the supports of two forms $\rho_i, \rho_j, i\ne j,$ intersect only in the  neighborhoods $\Nn_p$ in which
the $\rho_i$ are products, the form  $\om_{\ka,Y}'$ is nondegenerate provided that each form
$ \om +  \la_i \ka \rho_i$ is nondegenerate.  Therefore, we may take the lower bound $-\ka^0$ to be
$\max_i \frac {-\eps^0_i}{\la_i}$ and the upper bound   to be $\eps^1: = \min_i \frac{\eps_i^1}{\la_i}$.
This form satisfies (i) and (ii).   
Also, as explained in Step 3 of the proof of Lemma~\ref{le:tech1} the bound on $\ka^0$ depends on the size of the neighborhoods $\Nn_i$ of the curves $C^{T_i}$, and hence on
geometric information about $\Cc$ and $\om$.

If $Y\cdot T_i\ge 0$, for each $i$ the
quantity $\om_{\ka,Y}'(T_i)$ is a nondecreasing function of $\ka$.
But notice that as $\ka$ increases the area of $C^{T_i}$ is redistributed so
that \eqref{eq:Nn} eventually ceases to hold.
Thus when $\ka = \eps^1$ we isotop the form $\om^1: = \om_{\eps^1,Y}'$ near
$\Cc$ pushing area out of the product neighborhoods
to make \eqref{eq:Nn} valid again. Since $\eps_i^1$ is an increasing function of
 $\om_{\ka,Y}'(T_i)$ we may now repeat this process, starting with $\om^1$ and inflating
 by adding a suitable form in class $\ka PD(Y)$ for $\ka\in [0,\eps^2]$, where $\eps^2\ge \eps^1.$
 After a finite number of such steps, we arrive at a form in class $[\om_0] + \ka \PD(Y)$
 for arbitrarily large $\ka$. 
  If $Y\cdot T_i< 0$, for some $i$, then 
 $\om_{\ka,Y}'(T_i)$ decreases and it follows from  Lemma~\ref{le:tech1}
 that  the bound on $\ka^1$ 
 depends on cohomological data, namely $Y\cdot T_i$ and $\om(T_i)$.

 This gives a family of forms $\om_{\ka,Y}$ that satisfies (ii) and (iv), and nearly satisfies (i): the problem
 here is that we paused the inflation at $\ka = \eps^1, \eps^1+\eps^2$ and so on,
 while we readjusted the area distribution.  However, one can easily combine these two deformations
 and then reparametrize with respect to $\ka$ so as to satisfy (i).  Finally, note that when $Y\cdot T_j=0$ the total area of the curve $C^{T_j}$ is constant throughout the
 isotopy, although the distribution of area changes with $\ka$.  Hence 
 to achieve (iii) we alter the isotopy near each such component
 $C^{T_j}$ so that it is constant near that component. Again this is a standard Moser type argument: one should begin by
 adjusting the forms near each intersection point $C^{T_i}\cap C^{T_j}$, keeping the product structure, and then
 adjust near the rest of $\Cc$.
 \end{proof}

 \begin{proof}[Proof of Lemma~\ref{le:infl2}]
The proof of part (i) is similar, and will be left to the reader.
It uses the full force of Lemma~\ref{le:tech1}.
Moreover part (ii) holds because
at each step of the construction in Lemma~\ref{le:tech1}
the set of possible choices  (for example, of the size of the
neighborhoods $\Nn(C^{T_i})$ or of the precise normal form chosen for $\om_t$
as in \eqref{eq:omC1}) is contractible. Further, if one constructs two paths
$\om_{\ka^s}, s=0,1$, using the same
fibered structure (choice of projections $\pi_i$, radial coordinates $r$,
and neighborhoods $\Nn(C^{T_i}_t)$) then the linear isotopy
$$
(1-s)\om_{\ka,0} + s\om_{\ka,1},\quad 0\le s\le 1,
$$
between them consists of nondegenerate forms.
 \end{proof}

\bibliographystyle{alpha}

\end{document}